\documentclass[a4paper,10pt]{amsart}
\usepackage{amsmath}
\usepackage{amssymb}
\usepackage{latexsym}
\usepackage{amsbsy}
\usepackage[all]{xy}
\usepackage{amscd}
\usepackage[dvips]{graphicx}

\newtheorem{Prop}{Proposition}[section]
\newtheorem{Thm}[Prop]{Theorem}
\newtheorem{Lemma}[Prop]{Lemma}
\newtheorem{Cor}[Prop]{Corollary}

\newtheorem{Definition}[Prop]{Definition}

\def\az{\alpha}      \def\ud{{\underline{d}}}

\def\lz{\lambda}

\def\bbn{{\mathbb N}}    \def\bbq{{\mathbb Q}} \def\bb1{{\mathbb 1}}
\def\mv{{\mathcal V}}
   \def\bbe{{\mathbb E}} 

  \def\bbc{{\mathbb C}}

\def\ra{\rightarrow}
\def\hom{\mbox{Hom}}
\def\rad{\mbox{rad}\,}

\def\ext{\mbox{Ext}\,} 
\def\dim{\mbox{dim}\,}

\def\udim{\mbox{\underline {dim}}}

\def\mod{\mbox{mod}\,}  
\def\Im{\mbox{Im}\,}    \def\aut{\mbox{Aut}\,}
\def\ker{\mbox{Ker}\,}
\def\cok{\mbox{Coker}\,}

\def\lr#1{\langle #1\rangle}

\def\uq2{U_q(\hat{sl}_2)}

\def\bb{{\bf b}}

\def\nd{{\noindent}}

\def\md{{\mathcal{D}}}
\def\mp{{\mathcal{P}}}

\def\mo{{\mathcal{O}}}

\def\ue{{\underline{e}}}

\begin{document}

\title[Green's formula with $\bbc^{*}$-action and Caldero-Keller's formula]{Green's formula with $\bbc^{*}$-action and Caldero-Keller's formula for cluster algebras}

\thanks{ The research was
supported in part by NSF of China (No. 10631010) and by NKBRPC (No. 2006CB805905) \\
2000 Mathematics Subject Classification: 14M99,16G20, 16G70,
17B35. \\ Key words and phrases: Green's formula,  cluster
algebra, $\bbc^*$-action. }

\author{Jie Xiao and Fan Xu}
\address{Department of Mathematical Sciences\\
Tsinghua University\\
Beijing 10084, P.~R.~China} \email{jxiao@math.tsinghua.edu.cn
(J.Xiao),\  f-xu04@mails.tsinghua.edu.cn (F.Xu)}

\maketitle


\bigskip

\begin{abstract}
 It is known that Green's formula over finite fields gives rise to the comultiplications of
 Ringel-Hall algebras and quantum groups (see \cite{Green}, see
 also
\cite{Lusztig}). In this paper, we prove a projective version
 of  Green's formula in a geometric way. Then following the method of Hubery in \cite{Hubery2005},
 we apply this formula to proving  Caldero-Keller's multiplication formula for acyclic cluster
 algebras of arbitrary type.
\end{abstract}

\section{Introduction}

\subsection{}
Green in \cite{Green} found a homological counting formula for
hereditary abelian categories over finite fields. It leads to the
comultiplication formula for Ringel-Hall algebras, and as a
generalization of the result of Ringel in \cite{Ringel}, it gives
a realization of the positive part of the quantized enveloping
algebra for arbitrary type symmetrizable Kac-Moody algebra. In
\cite{DXX}, we gave Green's formula over the complex numbers
$\bbc$ via Euler characteristic and applied it to realizing
comultiplication of the universal enveloping algebra. However, one
should notice that many nonzero terms in the original formula
vanish when we consider it over the complex numbers $\bbc.$ In the
following, we show that the geometric correspondence in the proof
of Green's formula admits a canonical $\bbc^*$-action. Then we
obtain a new formula, which can be regarded as the projective
version of Green's formula.

Our motivation comes from cluster algebras. Cluster algebras were
introduced by Fomin and Zelevinsky \cite{FZ2002}. In \cite{BMRRT},
the authors categorified a lot of cluster algebras by defining and
studying the cluster categories related to clusters and seeds.
Under the framework of cluster categories, Caldero and Keller
realized the acyclic cluster algebras of simply-laced finite type
by proving a cluster multiplication theorem
\cite{CK2005}. At the same time, Hubery researched on realizing
acyclic cluster
algebras (including non simply-laced case) via Ringel-Hall
algebras for valued graphs over finite fields \cite{Hubery2005}.
He counted the corresponding Hall numbers and then deduced the
Caldero-Keller multiplication when evaluating at $q=1$ where $q$
is the order of the finite field.  It seems that his method only
works for the case of tame hereditary algebras \cite{Hubery2007},
due to the difficulty of the existence of Hall polynomials. In
this paper, we realize that the whole thing is independent of that
over finite fields. By counting the Euler characteristics of the
corresponding varieties and constructible sets with {\it
pushforward} functors and geometric quotients, we show that  the
projective version of Green's theorem and the ``higher order''
associativity of Hall multiplication
imply that Caldero-Keller's multiplication formula holds for
acyclic cluster algebras of arbitrary type. We remark here
that, for the elements in the dual semicanonical basis which are
given by certain constructible functions on varieties of nilpotent
modules over a preprojective algebra of arbitrary type, a similar
multiplication formula has been obtained in \cite{GLS}.

\subsection{}
The paper is organized as follows.  In Section \ref{basic}, we
recall the general theory of algebraic geometry needed in this
paper. This is followed in Section \ref{greenfinite} by a short
survey of Green's formula over finite fields without proof. In
particular, we consider many variants of Green's formula under
various group actions. These variants can be viewed as  the
counterparts over finite field of the projective version of
Green's formula. We give the main result in Section \ref{complex}.
Two geometric versions of Green's formula are proved. As an
application, in Section \ref{application} we prove Caldero-Keller
multiplication formula following Hubery's method \cite
{Hubery2005}, and also we give an example using the Kronecker
quiver.

\bigskip
\par\noindent {\bf Acknowledgments.}
The main idea of this paper comes from a sequence of discussion
 with Dr. Dong.Yang on the works \cite{CK2005} and \cite{Hubery2005}. We are grateful
 to Dr. D.Yang for his great help.

\section{Preliminaries}\label{basic}

\subsection{}
Let $Q=(Q_0,Q_1,s,t)$ be a quiver, where $Q_0$, also denoted by
$I$, and $Q_1$ are the
sets of vertices and arrows, respectively, and $s,t: Q_1\rightarrow
Q_0$ are
 maps such that any arrow $\az$ starts at $s(\az)$ and terminates at $t(\az).$ For any dimension vector
 $\ud=\sum_i a_i i\in\bbn I,$ we consider the affine space over $\bbc$
$$\bbe_{\ud}(Q)=\bigoplus_{\az\in Q_1}\hom_{\bbc}(\bbc^{a_{s(\az)}},\bbc^{a_{t(\az)}})$$
Any element $x=(x_{\az})_{\az\in Q_1}$ in $\bbe_{\ud}(Q)$ defines
a representation $M(x)$ with $\udim M(x)=\ud$ in a natural way.
For any $\alpha\in Q_1,$ we denote the vector space at $s(\alpha)$
(resp.$t(\alpha)$) of the representation $M$  by $M_{s(\alpha)}$
(resp.$M_{t(\alpha)}$) and  the linear map from $M_{s(\alpha)}$ to
$M_{t(\alpha)}$ by $M_{\alpha}.$  A relation in $Q$ is a linear
combination $\sum_{i=1}^{r}\lz_{i}p_i,$ where $\lz_i\in\bbc$ and
$p_i$ are paths of length at least two with $s(p_i)=s(p_j)$ and
$t(p_i)=t(p_j)$ for all $1\leq i,j\leq r.$ For any
$x=(x_{\az})_{\az\in Q_1}\in\bbe_{\ud}$ and any path
$p=\az_m\cdots\az_2\az_1$ in $Q$ we set
$x_{p}=x_{\az_m}\cdots x_{\az_2}x_{\az_1}.$  Then $x$ satisfies a
relation
 $\sum_{i=1}^{r}\lz_{i}p_i$ if $\sum_{i=1}^{r}\lz_i x_{p_i}=0.$ If $R$ is a set of relations in $Q,$ then
let $\bbe_{\ud}(Q,R)$ be the closed subvariety of $\bbe_{\ud}(Q)$
which consists of all elements satisfying all relations in $R.$ Any
element $x=(x_{\az})_{\az\in Q_1}$ in $\bbe_{\ud}(Q,R)$ defines in a
natural way a representation $M(x)$ of $A=\bbc Q/J$ with $\udim
M(x)=\ud,$ where $J$ is the admissible ideal generated by $R$.
 We consider the algebraic group
$$G_{\ud}(Q)=\prod_{\i\in I}GL(a_i,\bbc),$$ which acts on
 $\bbe_{\ud}(Q)$ by $(x_{\az})^{g}=(g_{t(\az)}x_{\az}g_{s(\az)}^{-1})$ for $g\in G_{\ud}$ and $(x_{\az})\in\bbe_{\ud}.$
It naturally induces an action of $G_{\ud}(Q)$ on
$\bbe_{\ud}(Q,R).$ The induced orbit space is denoted by
$\bbe_{\ud}(Q,R)/G_{\ud}(Q).$ There is a natural bijection between
the set ${\mathcal M}(A,\ud)$ of isomorphism classes of
$\bbc$-representations of $A$ with dimension vector $\ud$ and the
set of orbits of $G_{\ud}(Q)$ in $\bbe_{\ud}(Q,R).$ So we may
identify  ${\mathcal M}(A,\ud)$ with $\bbe_{\ud}(Q,R)/G_{\ud}(Q).$

The intersection of an open subset and a close subset in
$\bbe_{\ud}(Q,R)$ is called a locally closed subset. A subset in
$\bbe_{\ud}(Q,R)$ is called constructible if and only if it is a
disjoint union of finitely many locally closed subsets. Obviously,
an open set and a closed set are both constructible sets. A
function $f$ on $\bbe_{\ud}(Q,R)$ is called constructible if
$\bbe_{\ud}(Q,R)$ can be divided into finitely many constructible
sets such that $f$ is constant on each such constructible
set. Write M(X) for the $\mathbb{C}$-vector space of constructible
functions on some complex algebraic variety $X$.

Let $\mathcal{O}$ be a constructible set as defined above. Let
$1_{\mathcal{O}}$ be the characteristic function of $\mathcal{O}$,
defined by $1_{\mathcal{O}}(x)=1$, for any $x\in \mathcal{O}$ and
$1_{\mathcal{O}}(x)=0$, for any $x\notin \mathcal{O}$.
 It is clear that $1_{\mathcal{O}}$ is the simplest
constructible function and any constructible function is a linear
combination of characteristic functions.  For any constructible
subset $\mathcal{O}$ in $\bbe_{\ud}(Q,R)$, we call $\mathcal{O}$
$G_{\ud}$-invariant if $G_{\ud}\cdot\mathcal{O}=\mathcal{O}.$

In the following, we will always assume constructible sets and
functions to be $G_{\ud}$-invariant unless particular stated.

\subsection{}
Let $\chi$ denote Euler characteristic in compactly-supported
cohomology. Let $X$ be an algebraic variety and $\mo$ a
constructible subset which is the disjoint union of finitely many
locally closed subsets $X_i$ for $i=1,\cdots,m.$ Define
$\chi(\mo)=\sum_{i=1}^m\chi(X_i).$ Note that it is
well-defined. We will use the following properties:
\begin{Prop}[\cite{Riedtmann} and \cite{Joyce}]\label{Euler} Let $X,Y$ be algebraic varieties over $\mathbb{C}.$
Then
\begin{enumerate}
    \item  If the algebraic variety $X$ is the disjoint union of
finitely many constructible sets $X_1,\cdots,X_r$, then
$$\chi(X)=\sum_{i=1}^{r}{\chi(X_i)}.$$
    \item  If $\varphi:X\longrightarrow Y$ is a morphism
with the property that all fibers have the same Euler
characteristic $\chi$, then $\chi(X)=\chi\cdot \chi(Y).$ In
particular, if $\varphi$ is a locally trivial fibration in the
analytic topology with fibre $F,$ then $\chi(X)=\chi(F)\cdot
\chi(Y).$
    \item $\chi(\bbc^n)=1$ and $\chi(\mathbb{P}^n)=n+1$ for all $n\geq
    0.$
\end{enumerate}
\end{Prop}
We recall the definition {\it pushforward} functor from the
category of algebraic
varieties over $\mathbb{C}$ to the category of $\mathbb{Q}$-vector
spaces (see \cite{Macpherson} and \cite{Joyce}).

\nd Let $\phi: X\rightarrow Y$ be a morphism of varieties. For $f\in
M(X)$ and $y\in Y,$ define
$$
\phi_{*}(f)(y)=\sum_{c\in\bbq}c\chi(f^{-1}(c)\cap \phi^{-1}(y))
$$
\begin{Thm}[\cite{Dimca},\cite{Joyce}]\label{Joyce}
Let $X,Y$ and $Z$ be algebraic varieties over $\mathbb{C},$ $\phi:
X\rightarrow Y$ and $\psi: Y\rightarrow Z$ be morphisms of
varieties, and $f\in M(X).$ Then $\phi_{*}(f)$ is constructible,
$\phi_{*}: M(X)\rightarrow M(Y)$ is a $\mathbb{Q}$-linear map and
$(\psi\circ \phi)_{*}=(\psi)_{*}\circ (\phi)_{*}$ as
$\mathbb{Q}$-linear maps from $M(X)$ to $M(Z).$
\end{Thm}

 In order to deal with orbit spaces, we
 need to consider
geometric quotients.
\begin{Definition}\label{quotient}
Let $G$ be an algebraic group acting on a variety $X$ and
$\phi:X\rightarrow Y$ be a $G$-invariant morphism, i.e. a morphism
constant on orbits. The pair $(Y,\phi)$ is called a geometric
quotient if $\phi$ is open and for any open subset $U$ of $Y$, the
associated comorphism identifies the ring $\mo_{Y}(U)$ of regular
functions on $U$ with the ring $\mo_{X}(\phi^{-1}(U))^{G}$ of
$G$-invariant regular functions
 on $\phi^{-1}(U)$.
\end{Definition}

The following result due to Rosenlicht \cite{Ro} is essential to us.

\begin{Lemma}\label{Rosenlicht}
Let $X$ be a $G$-variety, then there exists an  open and dense
$G$-stable subset which has a geometric $G$-quotient.
\end{Lemma}
By this Lemma, we can construct a finite stratification over $X.$
Let $U_1$ be an
open and dense $G$-stable subset of $X$ as in Lemma
\ref{Rosenlicht}. Then
$\dim_{\mathbb{C}}(X-U_1)<\dim_{\mathbb{C}}X.$ We can use the
above
lemma again, there exists a dense open $G$-stable subset $U_2$ of
$X-U_1$ which has a geometric $G$-quotient. Inductively, we get a
finite stratification
$X=\cup_{i=1}^{l}U_{i}$ where $U_{i}$ is a
$G$-invariant locally closed subset and has a geometric quotient,
$l\leq \dim_{\mathbb{C}}X.$ We denote by $\phi_{U_i}$ the
geometric
quotient map on $U_i.$ Define the quasi Euler-Poincar\'e
characteristic of $X/G$ by
$\chi(X/G):=\sum_{i}\chi(\phi_{U_i}(U_i)).$ If $\{U'_i\}$ is
another choice in the definition of $\chi(X/G)$, then
$\chi(\phi_{U_i}(U_i))=\sum_{j}\chi(\phi_{U_i\cap U'_j}(U_i\cap
U'_j))$ and $\chi(\phi_{U'_j}(U'_j))=\sum_{i}\chi(\phi_{U_i\cap
U'_j}(U_i\cap U'_j)).$ Thus
$\sum_{i}\chi(\phi_{U_i}(U_i))=\sum_{i}\chi(\phi_{U'_i}(U'_i))$
and $\chi(X/G)$ is well-defined (see \cite{XXZ2006}). Similarly,
$\chi(\mo/G):=\sum_i\chi(\phi_{U_i}(\mo\bigcap U_i))$ is
well-defined for any $G$-invariant constructible subset $\mo$ of
$X.$

\subsection{}
We also introduce the following notation. Let $f$ be a
constructible function over a variety $X,$ it is natural to define
\begin{equation}\label{integral}
\int_{x\in X}f(x):=\sum_{m\in \bbc}m\chi(f^{-1}(m))
\end{equation}
Comparing with Proposition \ref{Euler}, we also have the following
(see \cite{XXZ2006}).
\begin{Prop}\label{Euler2} Let $X,Y$ be algebraic varieties over
$\mathbb{C}$ under the actions of the algebraic groups $G$ and
$H$,
%
respectively.  Then
\begin{enumerate}
    \item  If the algebraic variety $X$ is the disjoint union of
finitely many $G$-invariant constructible sets $X_1,\cdots,X_r$,
then
$$\chi(X/G)=\sum_{i=1}^{r}{\chi(X_i/G)}$$
    \item  If a morphism $\varphi:X\longrightarrow Y$  induces
    a quotient map $\phi:X/G\rightarrow Y/H$
whose fibers all have the same Euler
characteristic $\chi$, then $\chi(X/G)=\chi\cdot \chi(Y/H).$
   \end{enumerate}
\end{Prop}

Moreover, if there exists an action of an algebraic group $G$ on
$X$
as in Definition \ref{quotient}, and $f$ is a $G$-invariant
constructible function over $X,$ we define
\begin{equation}\label{integral2}
\int_{x\in X/G}f(x):=\sum_{m\in \bbc}m\chi(f^{-1}(m)/G)
\end{equation}

In particular, we frequently use the following corollary.
\begin{Cor}\label{euler3}
Let $X,Y$ be algebraic varieties over $\mathbb{C}$ under the actions
%
of an algebraic group $G.$  These actions naturally induce an
action of $G$ on $X\times Y.$ Then
$$\chi(X\times_{G}Y)=\int_{y\in Y/G}\chi(X/G_{y})$$
where $G_y$ is the stabilizer in $G$ of $y\in Y$ and
$X\times_{G}Y$ is the orbit space of $X\times Y$ under the action of
$G.$
\end{Cor}

\section{Green's formula over finite fields}\label{greenfinite}

\subsection{}
In this section, we recall Green's formula  over finite fields
(\cite{Green},\cite{RingelGreen}). Let $k$ be a finite field and
$\Lambda$ a hereditary finitary
 $k$-algebra, i.e., $\ext^1(M,N)$ is a finite set and $\ext^2(M,N)=0$ for any $\Lambda$-modules $M, N$.
 Let $\mathcal{P}$ be
the set of isomorphism classes of finite $\Lambda$-modules. Let
$\mathcal{H}(\Lambda)$ be the Ringel-Hall algebra associated to
$\mod\Lambda.$  Green introduced on $\mathcal{H}$ a
comultiplication so that $\mathcal{H}$ becomes a bialgebra up to a
twist on $\mathcal{H}\bigotimes\mathcal{H}.$ His proof of the
compatibility between the multiplication and the comultiplication
completely depends on the following Green's formula.

Given $\alpha\in \mathcal{P},$ let $V_{\alpha}$ be a
representative in $\alpha,$ and $a_{\alpha}=|\aut_{\Lambda}
V_{\alpha}|.$ Given $\xi,\eta$ and $\lambda$ in $\mathcal{P},$ let
$g_{\xi\eta}^{\lambda}$ be the number of submodules $Y$ of
$V_{\lambda}$ such that $Y$ and $V_{\lambda}/Y$ belong to
$\eta$ and $\xi$, respectively.
\begin{Thm}
Let $k$ be a finite field and $\Lambda$ a hereditary finitary
$k$-algebra. Let $\xi,\eta,\xi',\eta'\in \mathcal{P}.$ Then
$$
a_{\xi}a_{\eta}a_{\xi'}a_{\eta'}\sum_{\lambda}g_{\xi\eta}^{\lambda}g_{\xi'\eta'}^{\lambda}a_{\lambda}^{-1}
=\sum_{\alpha,\beta,\gamma,\delta}\frac{|\mathrm{Ext}
^1(V_{\gamma},V_{\beta})|}{|\mathrm{Hom}
(V_{\gamma},V_{\beta})|}g_{\gamma\alpha}^{\xi}
g_{\gamma\delta}^{\xi'}g_{\delta\beta}^{\eta}g_{\alpha\beta}^{\eta'}a_{\alpha}a_{\beta}a_{\delta}a_{\gamma}
$$
\end{Thm}
 Suppose $X\in \xi, Y\in \eta,
M\in\xi', N\in \eta'$ and $A\in \gamma, C\in \alpha, B\in \delta,
D\in \beta, E\in\lambda.$
Set $h^{\xi\eta}_{\lambda}:=|\ext^1(X,Y)_{E}|,$ where
$\ext^1(X,Y)_{E}$ is the subset of $\ext^1(X,Y)$ consisting of
elements $\omega$ such that the middle term of an exact sequence
represented by $\omega$
is isomorphic to $E.$ Then the above formula can be rewritten as
(\cite{DXX},\cite{Hubery2005})
$$
\sum_{\lambda}g_{\xi\eta}^{\lambda}h^{\xi'\eta'}_{\lambda}=\sum_{\alpha,\beta,\gamma,\delta}\frac{|\ext^1(A,D)|
|\hom(M,N)|}{|\hom(A,D)||\hom(A,C)||\hom(B,D)|}g_{\gamma\delta}^{\xi'}g_{\alpha\beta}^{\eta'}h^{\gamma\alpha}_{\xi}
h^{\delta\beta}_{\eta}
$$

\subsection{}
For fixed $kQ$-modules $X,Y,M,N$ with $\udim X+\udim Y=\udim
M+\udim N,$ we fix a
$Q_0$-graded $k$-space $E$ such that $ \udim E=\udim X+\udim Y.$
Let $(E,m)$ be the
$kQ$-module structure on $E$ given by an
algebraic morphism $m: \Lambda\rightarrow \mathrm{End}
_kE.$ Let $Q(E,m)$ be the set of $(a,b,a',b')$ such that the row
and the column of the following diagram are exact:
\begin{equation}\label{E:crosses}
\xymatrix{
& & 0 \ar[d] & &\\
& & Y \ar[d]^-{a'} & & \\
0 \ar[r] & N \ar[r]^-{a} & (E,m) \ar[r]^{b} \ar[d]^-{b'} & M \ar[r] & 0\\
& & X \ar[d] & &\\
& & 0 & &}
\end{equation}
Let
$$
Q(X,Y,M,N)=\bigcup_{m:\Lambda\rightarrow End
_kE}Q(E,m)
$$
It is clear that
$$
|Q(E,m)|=g_{\xi\eta}^{\lambda}g_{\xi'\eta'}^{\lambda}a_{\xi}a_{\eta}a_{\xi'}a_{\eta'}
$$ where $\lz\in\mathcal{P}$ is such that
$(E,m)\in\lz,$ or simply write
$m\in\lz.$
$$
|Q(X,Y,M,N)|=\sum_{\lambda}\frac{|Aut_{k}E|}{a_{\lambda}}g_{\xi\eta}^{\lambda}g_{\xi'\eta'}^{\lambda}a_{\xi}a_{\eta}a_{\xi'}a_{\eta'}
$$
There is an action of $\aut_{\Lambda}(E,m)$ on $Q(E,m)$ given by

$$g.(a,b,a',b')=(ga,bg^{-1},ga',b'g^{-1})$$
This induces an orbit space of $Q(E,m)$, denoted by $Q(E,m)^{*}$.
The orbit of $(a,b,a',b')$ in $Q(E,m)^{*}$
is denoted by $(a,b,a',b')^{*}$. We have
$$
|Q(X,Y,M,N)|=|\aut_{k}E|\sum_{\lambda\in
\mp}\sum_{(a,b,a',b')^{*}\in Q(E,m)^{*},m\in
\lambda}\frac{1}{|\hom(\cok b'a,\ker ba')|}
$$
Furthermore, there is an action of the group $\aut X\times \aut Y$
on
$Q(E,m)^{*}$ given by
$$
(g_1,g_2).(a,b,a',b')^{*}=(a,b,a'g_2^{-1},g_1b')^{*}
$$
for $(g_1,g_2)\in \aut X\times \aut Y$ and $(a,b,a',b')^{*}\in
Q(E,m)^{*}.$ The stabilizer $G((a,b,a',b')^{*})$ of
$(a,b,a',b')^{*}$ is
$$
\{(g_1,g_2)\in \aut X\times \aut Y\mid ga'=a'g_2,b'g=g_1b'\quad
\mbox{for some }g\in 1+a\hom(M,N)b\}
$$

The orbit space is denoted by $Q(E,m)^{\wedge}$ and the orbit of
$(a,b,a',b')^{*}$ is
denoted by $(a,b,a',b')^{\wedge}.$ We have
$$
\frac{1}{a_Xa_Y}|Q(E,m)^{*}|=\sum_{(a,b,a',b')^{\wedge}\in
Q(E,m)^{\wedge}}\frac{1}{|G((a,b,a',b')^{*})|}
$$

\subsection{}
Let $\md(X,Y,M,N)^{*}$ be the set of $(B,D,e_1,e_2,e_3,e_4)$ such
that the following diagram has exact rows and exact
columns:
\begin{equation}\label{E:gurz}
\xymatrix{ & 0 \ar[d] &&&&  0\ar[d] && \\
0 \ar[r] & D \ar[rr]^-{e_1} \ar[dd]^-{u'} && Y\ar[rr]^-{e_2} && B \ar[r] \ar[dd]^-{x}& 0\\
&&&&&&\\
& N \ar[dd]^-{v'} && && M \ar[dd]^-{y} &\\
&&&&&&\\
0\ar[r] &  C \ar[d] \ar[rr]^-{e_3} && X \ar[rr]^-{e_4} && A \ar[r] \ar[d] & 0\\
& 0 && && 0 &}
\end{equation}
where $B,D$ are submodules of $M,N,$ respectively and $A=M/B,
C=N/D.$ The maps $u',v'$ and $x,y$ are naturally induced. We have
$$
|\md(X,Y,M,N)^{*}|=\sum_{\alpha,\beta,\gamma,\delta}g_{\gamma\alpha}^{\xi}g_{\gamma\delta}^{\xi'}g_{\delta\beta}^{\eta}g_{\alpha\beta}^{\eta'}a_{\alpha}a_{\beta}a_{\delta}a_{\gamma}
$$

There is an action of the group $\aut_{\Lambda}X\times
\aut_{\Lambda}Y$ on $\md(X,Y,M,N)^{*}$ given by
$$
(g_1,g_2).(B,D,e_1,e_2,e_3,e_4)=(B,D,g_2e_1,e_2g_2^{-1},g_1e_3,e_4g_1^{-1})
$$
for $(g_1,g_2)\in \aut_{\Lambda}X\times \aut_{\Lambda}Y.$ The
orbit space is denoted by $\md(X,Y,M,N)^{\wedge}.$ We have
$$
|\md(X,Y,M,N)^{\wedge}|=\frac{1}{a_Xa_Y}\sum_{\alpha,\beta,\gamma,\delta}|\hom(A,C)|
|\hom(B,D)|g_{\gamma\alpha}^{\xi}g_{\gamma\delta}^{\xi'}g_{\delta\beta}^{\eta}g_{\alpha\beta}^{\eta'}a_{\alpha}
a_{\beta}a_{\delta}a_{\gamma}
$$

Fix a square as above ,
let $T=X \times_{A} M =\{(x \oplus m) \in X
\oplus M\;|\; e_4(x)=y(m)\}$ and $S=Y \sqcup_{\small{D}}
\hspace{-.05in}N =Y \oplus N / \{e_1(d)\oplus u'(d)\;|\;d \in D\}.$
There is a unique map $f:S\rightarrow T$ ( see \cite{RingelGreen})
such that the natural long sequence
\begin{equation}\label{longexactsequence}
0\rightarrow D\rightarrow S\stackrel{f}\rightarrow T\rightarrow
A\rightarrow 0
\end{equation}
is exact.

Let $(c,d)$ be a pair of maps such that $c$ is
surjective, $d$ is injective and $cd=f.$ The number of such pairs
can be computed as follows.
We have the following
commutative diagram:
\begin{equation}\label{3}
\xymatrix{&0\ar[r]& S\ar[r]^{d} \ar[d]& (E,m)
\ar[r]^{d_1}\ar[d]^{c}&
A\ar[r]\ar@{=}[d]&0\\
\varepsilon_0:& 0\ar[r] & \mathrm{Im}f
 \ar[r]& T\ar[r] &A \ar[r]& 0}
\end{equation}
The exact sequence
$$
\xymatrix{0\ar[r]&D\ar[r]& S\ar[r]& \mathrm{Im}f
 \ar[r]&0}
$$
induces the following long exact sequence:
\begin{equation}\label{E:Cross4}
\xymatrix{ 0 \ar[r] & \text{Hom}(A,D) \ar[r] & \text{Hom}(A,S)
\ar[r]& \text{Hom}(A,\mathrm{Im}f
) \ar[r] &}
\end{equation}
$$\qquad\xymatrix{
 \ar[r] & \text{Ext}^1(A,D) \ar[r] & \text{Ext}^1(A,S) \ar[r]^-{\phi} & \text{Ext}^1(A,\mathrm{Im}f
) \ar[r] & 0}
$$
We set $\varepsilon_0\in \ext^1(A,\Im f)$ corresponding to the
canonical exact sequence
$$
\xymatrix{0\ar[r]&\Im f \ar[r]& T \ar[r]& A \ar[r]&0}
$$
and denote $\phi^{-1}(\varepsilon_0)\bigcap \ext^1(A,S)_{(E,m)}$
by $\phi^{-1}_{m}(\varepsilon_0).$
 Let $\mathcal{F}(f;m)$ be the set of
$(c,d)$ induced by diagram \eqref{3} with center term $(E,m).$ Let
$$\mathcal{F}(f)=\bigcup_{m:\Lambda\rightarrow
\mathrm{End}
_kE}\mathcal{F}(f;m)$$  Then
$$
|\mathcal{F}(f;m)|
=|\phi^{-1}_{m}(\varepsilon_0)|\frac{|\aut_{\Lambda}(E,m)|}{|\hom(A,S)|}|\hom(A,\mathrm{Im}f
)|,
$$
$$
|\mathcal{F}(f)|
=|\aut_{k}(E)|\frac{|\ext^1(A,D)|}{|\hom(A,D)|}.
$$

Let $\mo(E,m)$ be the set of $(B,D,e_1,e_2,e_3,e_4,c,d)$ such that
the following diagram is commutative and has exact rows and
columns:
\begin{equation}\label{5}
\xymatrix{ & 0 \ar[d] &&&&  0\ar[d] && \\
0 \ar[r] & D \ar[rr]^-{e_1} \ar[dd]^-{u'} && Y\ar[rr]^-{e_2} \ar@{.>}[dl]^-{u_Y}\ar@{.>}[dd]&& B \ar[r] \ar[dd]^-{x}& 0\\
&&S\ar@{.>}[dr]^-{d}&&&&\\
& N \ar@{.>}[ur]_-{u_N}\ar[dd]^-{v'} \ar@{.>}[rr]&& (E,m)\ar@{.>}[rr]\ar@{.>}[dd]\ar@{.>}[dr]^-{c}&& M \ar[dd]^-{y} &\\
&&&&T\ar@{.>}[ur]_-{q_M}\ar@{.>}[dl]^-{q_X}&&\\
0\ar[r] &  C \ar[d] \ar[rr]^-{e_3} && X \ar[rr]^-{e_4} && A \ar[r] \ar[d] & 0\\
& 0 && && 0 &}
\end{equation}
where the maps $q_X,u_Y$ and $q_M,u_N$
are naturally induced. In fact, the long exact sequence
\eqref{longexactsequence} has the following explicit form:
\begin{equation}\label{longexactsequence2}
\xymatrix{0\ar[r]&
D\ar[r]^{u_Ye_1}&S\ar[r]^{cd}&T\ar[r]^{e_4q_X}&A\ar[r]& 0}
\end{equation}

$$
|\mo(E,m)|=\sum_{\alpha,\beta,\gamma,\delta}|\phi^{-1}_{m}(\varepsilon_0)|\frac{|\aut_{\Lambda}(E,m)|}{|\hom(A,S)|}|\hom(A,\mathrm{Im}
f)|g_{\gamma\alpha}^{\xi}g_{\gamma\delta}^{\xi'}g_{\delta\beta}^{\eta}g_{\alpha\beta}^{\eta'}a_{\alpha}a_{\beta}a_{\delta}a_{\gamma}
$$
Let $\mo(X,Y,M,N)=\bigcup_{m:\Lambda\rightarrow \mathrm{End}
_kE}\mo(E,m)$
$$
|\mo(X,Y,M,N)|=|\aut_{k}(E)|\sum_{\alpha,\beta,\gamma,\delta}\frac{|\ext^1(A,D)|}{|\hom(A,D)|}g_{\gamma\alpha}^{\xi}g_{\gamma\delta}^{\xi'}g_{\delta\beta}^{\eta}g_{\alpha\beta}^{\eta'}a_{\alpha}a_{\beta}a_{\delta}a_{\gamma}
$$

The group $\aut_{\Lambda} (E,m)$ naturally acts on
$\mo(E,m)$ and $\mo(X,Y,M,N)$ as follows:
$$
g.(B,D,e_1,e_2,e_3,e_4,c,d)=(B,D,e_1,e_2,e_3,e_4,cg^{-1},gd)
$$
We denote the orbit space by $\mo(E,m)^{*}$ and
$\mo(X,Y,M,N)^{*}.$ The orbit of $(B,D,e_1,e_2,e_3,e_4,c,d)$ is
denoted by $(B,D,e_1,e_2,e_3,e_4,c,d)^{*}.$ Then
$$
|\mo(X,Y,M,N)^{*}|=\sum_{\alpha,\beta,\gamma,\delta}|\ext^1(A,D)|g_{\gamma\alpha}^{\xi}g_{\gamma\delta}^{\xi'}g_{\delta\beta}^{\eta}g_{\alpha\beta}^{\eta'}a_{\alpha}a_{\beta}a_{\delta}a_{\gamma}
$$
Similar to that on $\md(X,Y,M,N)^{*},$ there is an action of $\aut
X\times
\aut Y$ on $\mo(X,Y,M,N)^{*}$ given by
$$
(g_1,g_2).{(B,D,e_1,e_2,e_3,e_4,c,d)}^{*}=(B,D,g_2e_1,e_2g_2^{-1},g_1e_3,e_4g_1^{-1},c',d')^{*}
$$
Let us determine the relation between $(c',d')$ and $(c,d).$

It is clear that there are isomorphisms:
$$a_1: S\rightarrow S' \quad \mbox{and} \quad a_2: T\rightarrow T'$$
induced by isomorphisms:
$$
\left(%
\begin{array}{cc}
  g_2 & 0\\
  0& id \\
\end{array}%
\right): Y\oplus N\rightarrow Y\oplus N\quad \mbox{and} \quad \left(%
\begin{array}{cc}
  g_1 & 0\\
  0& id \\
\end{array}%
\right): X\oplus M\rightarrow X\oplus M
$$
Hence, $c'=a_2c,d'=da_1^{-1}.$

The stabilizer of
 $(B,D,e_1,e_2,e_3,e_4,c,d)^{*}$ is denoted
by $G((B,D,e_1,e_2,e_3,e_4,c,d)^{*}),$ which is
\begin{eqnarray}
  \{(g_1,g_2)\in \aut X\times \aut Y
 &\mid& g_1\in e_3\hom(A,C)e_4, g_2\in e_1\hom(B,D)e_2, \nonumber\\
   && cg^{-1}=c',gd=d' \mbox{ for some }g\in \aut (E,m)\}\nonumber
\end{eqnarray}

The orbit space is denoted by $\mo(X,Y,M,N)^{\wedge}$ and the
orbit is denoted by $(B,D,e_1,e_2,e_3,e_4,c,d)^{\wedge}.$ Then
$$
\frac{1}{a_Xa_Y}|\mo(X,Y,M,N)^{*}|=\sum_{(B,D,e_1,e_2,e_3,e_4,c,d)^{\wedge}\in
\mo(X,Y,M,N)^{\wedge}}\frac{1}{|G((B,D,e_1,e_2,e_3,e_4,c,d)^{*})|}
$$

\subsection{}
There is a bijection $\Omega:Q(E,m)\rightarrow \mo(E,m)$ which
induces Green's formula.
 In the same way, we also have the
following proposition
\begin{Prop}
There exist bijections $\Omega^{*}: Q(E,m)^{*}\rightarrow
\mo(E,m)^{*}$ and $\Omega^{\wedge}: Q(E,m)^{\wedge}\rightarrow
\mo(E,m)^{\wedge}.$
\end{Prop}
\begin{proof}
For any $(a,b,a',b')\in Q(E,m),$
$$
\Omega(a,b,a',b')=(\ker b'a,\Im ba',a'^{-1}a, ba',b'a,
bb'^{-1},c,d)
$$
where $c,d$ are defined by
$$
du_N=a,du_Y=a',q_Mc=b,q_Xc=b'
$$
Hence,
\begin{eqnarray}
  \Omega
(g.(a,b,a',b')) &=& (ga,bg^{-1},ga',b'g^{-1}) \nonumber \\
   &=&  (\ker b'a,\Im ba',a'^{-1}a,
ba',b'a, bb'^{-1},cg^{-1},gd)\nonumber \\
  &=& g.(\ker b'a,\Im ba',a'^{-1}a, ba',b'a, bb'^{-1},c,d)\nonumber
\end{eqnarray}
i.e.
$$
\Omega^{*}((a,b,a',b')^{*})=((\ker b'a,\Im ba',a'^{-1}a, ba',b'a,
bb'^{-1},c,d))^{*}
$$
for $g\in \aut(E,m).$ Similarly,
\begin{eqnarray}
  \Omega^{*}
((g_1,g_2).(a,b,a',b')^{*}) &=& (a,b,a'g_2^{-1},g_1b')^{*} \nonumber \\
   &=&  (\ker g_1b'a,\Im ba'g_2^{-1},g_2a'^{-1}a,
ba'g_2^{-1},g_1b'a, bb'^{-1}g_1^{-1},c',d')^{*}\nonumber \\
  &=& (g_1,g_2).(\ker b'a,\Im ba',a'^{-1}a, ba',b'a, bb'^{-1},c,d)^{*}\nonumber
\end{eqnarray}
for  $(g_1,g_2)\in \aut X\times \aut Y.$ Hence
$$
\Omega^{\wedge}((a,b,a',b')^{\wedge})=((\ker b'a,\Im ba',a'^{-1}a,
ba',b'a, bb'^{-1},c,d))^{\wedge}
$$
\end{proof}

In particular, if $(a,b,a',b')^{*}$ corresponds to
$(B,D,e_1,e_2,e_3,e_4,c,d)^{*},$ then
$$
G((a,b,a',b')^{*})=G((B,D,e_1,e_2,e_3,e_4,c,d)^{*})
$$
 We also give the following variant of Green's formula, which is
 suggestive for the  projective Green's formula over the complex numbers in the next section.

\begin{eqnarray*}
 \lefteqn{\sum_{\lambda; \lambda\neq \xi'\oplus
\eta'}\frac{1}{q-1}h^{\xi'\eta'}_{\lambda}g_{\xi\eta}^{\lambda}=}\\
   && \sum_{\alpha,\beta,\gamma,\delta;\alpha\oplus \gamma\neq \xi
\mbox{ or }\beta\oplus \delta\neq
\eta}\frac{|\ext^1(A,D)||\hom(M,N)|}{|\hom(A,D)||\hom(A,C)||\hom(B,D)|}\frac{1}{q-1}h^{\gamma\alpha}_{\xi}h^{\delta\beta}_{\eta}g_{\gamma\delta}^{\xi'}g_{\alpha\beta}^{\eta'}\\
   && +\sum_{\alpha,\beta,\gamma,\delta;\alpha\oplus
\gamma=\xi,\beta\oplus \delta=
\eta}\frac{1}{q-1}(\frac{|\ext^1(A,D)||\hom(M,N)|}{|\hom(A,D)||\hom(A,C)||\hom(B,D)|}-1)g_{\gamma\delta}^{\xi'}g_{\alpha\beta}^{\eta'}\\
   && +\frac{1}{q-1}(\sum_{\alpha,\beta,\gamma,\delta;\alpha\oplus
\gamma=\xi,\beta\oplus \delta=
\eta}g_{\gamma\delta}^{\xi'}g_{\alpha\beta}^{\eta'}-g_{\xi\eta}^{\xi'\oplus
\eta'})
\end{eqnarray*}

\bigskip

\section{Green's formula over the complex numbers}\label{complex}

\subsection{}
From now on, we consider  $A=\bbc Q,$ where $\bbc$ is the field of
complex
numbers. Let $\mo_1,\mo_2$ be $G$-invariant constructible
subsets in $\bbe_{\ud_1}(Q),\bbe_{\ud_2}(Q),$ respectively, and
let $\ud = \ud_1 + \ud_2$.
Define
$$\mv(\mathcal{O}_1,\mathcal{O}_2;L)=\{0=X_0\subseteq
X_1\subseteq X_2=L\mid X_i\in \mod
\ A,X_1 \in \mathcal{O}_2,
 \mbox{  and  } L/X_1\in \mathcal{O}_1\}.$$
where $L\in \bbe_{\ud}(Q).$  In particular, when $\mo_1,\mo_2$ are
the orbits of $A$-modules $X,Y$ respectively, we write $\mv(X,Y;L)$
instead of $\mv(\mo_1,\mo_2;L).$

Let $\alpha$ be the image of $X$  in
$\bbe_{\ud_{\alpha}}(Q)/G_{\ud_{\alpha}}.$  We write $X\in \alpha$,
sometimes we also use the notation $\overline{X}$ to denote the
image of $X$ and the notation $V_{\alpha}$ to denote a
representative of $\alpha.$ Instead of $\ud_{\alpha},$ we use
$\underline{\alpha}$ to denote the dimension vector of $\alpha.$ Put
$$
g_{\alpha\beta}^{\lambda}=\chi(\mv(X,Y;L))
$$ for $X\in\alpha, Y\in\beta$  and $L\in\lambda.$
Both are well-defined and independent of  the choice of objects in
the
%
orbits.
\begin{Definition}\cite{Riedtmann}\label{grassmannian}
For any $L\in \mathrm{mod}\
A,$ let $L=\bigoplus_{i=1}^{r}L_i$ be the
decomposition into indecomposables, then an action of $\bbc^{*}$
on $L$ is defined by
$$
t.(v_1,\cdots,v_r)=(tv_1,\cdots,t^rv_r)
$$
for $t\in \bbc^{*}$ and $v_i\in L_i$ for
$i=1,\cdots,r.$

It induces an action of $\bbc^{*}$ on $\mv(X,Y;L)$ for any
$A$-modules
$X,Y$ and $L.$ Let $(X_1\subseteq L)\in \mv(X,Y;L)$ and $t.X_1$ be
the action of $\bbc^{*}$ on $X_1$ as above under the decomposition
of
$L,$ then there is a natural  isomorphism between $A$-modules
$t_{X_1}: X_1\simeq t.X_1.$ Define $t.(X_1\subseteq
L)=(t.X_1\subseteq L).$
\end{Definition}

Let $D(X,Y)$ be the vector space over $\bbc$ of all tuples
$d=(d(\alpha))_{\alpha\in Q_1}$ such that each linear map
$d(\alpha)$ belongs to
$\mathrm{Hom}_{\bbc}(X_{s(\alpha)},Y_{t(\alpha)}).$
Define $\pi: D(X,Y)\rightarrow \mathrm{Ext}^1(X,Y)$ by sending $d$
to the short exact sequence
$$
\xymatrix{\varepsilon:\quad 0\ar[r]& Y\ar[rr]^{\left(%
\begin{array}{c}
  1 \\
  0 \\
\end{array}%
\right)}&&L(d)\ar[rr]^{\left(%
\begin{array}{cc}
  0 & 1 \\
\end{array}%
\right)}&&X\ar[r]&0}
$$
where $L(d)$
is the direct sum of $X$ and $Y$ as a vector space and for
any $\alpha\in Q_1,$
$$
L(d)_{\alpha}=\left(%
\begin{array}{cc}
  Y_{\alpha} & d(\alpha) \\
  0 & X_{\alpha} \\
\end{array}%
\right)
$$
Fix
a vector space decomposition $D(X,Y)=\mathrm{Ker}\pi\oplus
E(X,Y),$ then we can identify $\mathrm{Ext}^{1}(X,Y)$ with
$E(X,Y)$ ( see \cite{Riedtmann}, \cite{DXX} or \cite{GLS} ). There
is a natural $\bbc^{*}$-action on $E(X,Y)$ given by
$t.d=(td(\alpha))$ for any $t\in \bbc^{*}.$ This induces an action
of $\bbc^{*}$ on $\mathrm{Ext}^{1}(X,Y).$ By the isomorphism of
$\bbc Q$-modules between $L(d)$ and $L(t.d),$ $t.\varepsilon$ is
the
following short exact sequence:
$$
\xymatrix{0\ar[r]& Y\ar[rr]^{\left(%
\begin{array}{c}
  t \\
  0 \\
\end{array}%
\right)}&&L(d)\ar[rr]^{\left(%
\begin{array}{cc}
  0 & 1 \\
\end{array}%
\right)}&&X\ar[r]&0}
$$
for any $t\in \bbc^{*}.$ Let $\mathrm{Ext}^{1}(X,Y)_{L}$ be the
subset of $\mathrm{Ext}^{1}(X,Y)$ of the equivalence classes of
short exact sequences whose middle term is isomorphic
to $L$. Then
$\mathrm{Ext}^{1}(X,Y)_{L}$ can be viewed as a constructible
subset of $\mathrm{Ext}^1(X,Y)$ under the
identification between $\mathrm{Ext}^1(X,Y)$ and $E(X,Y).$ Put
$$
h_{\lambda}^{\alpha\beta}=\chi(\mathrm{Ext}^{1}_{A}(X,Y)_{L})
$$
for $X\in \alpha, Y\in \beta$ and $L\in \lambda.$ The following is
known, for example, see \cite{DXX}.
\begin{Lemma}\label{directsum}
For $A,B,X\in \mathrm{mod}
\Lambda,$ $\chi(\mathrm{Ext}
_{\Lambda}^{1}(A,B)_{X})=0$
unless $X\simeq A\oplus B.$
\end{Lemma}
We remark that both $\mv(X,Y;L)$ and $\mathrm{Ext}
^1(X,Y)_{L}$ can be viewed
as the orbit spaces of
$$
W(X,Y;L):=\{(f,g)\mid \xymatrix{0\ar[r]&Y\ar[r]^{f}&L\ar[r]^{g}&
X\ar[r]& 0}~~\mbox{is an exact sequence}\}
$$
under the actions of $G_{\underline{\alpha}}\times
G_{\underline{\beta}}$ and $G_{\underline{\lambda}}$ respectively,
for $X\in \alpha, Y\in \beta$ and $L\in \lambda.$

\bigskip
\subsection{}
 For fixed $\xi, \eta, \xi', \eta',$ consider the following
canonical embedding:
\begin{equation}\label{decom}
\bigcup_{\alpha,\beta,\gamma,\delta;\alpha\oplus
\gamma=\xi,\beta\oplus \delta= \eta}\mv(V_{\alpha},V_{\beta};
V_{\eta'})\times \mv(V_{\gamma},V_{\delta};V_{\xi'})\xrightarrow{i}
\mv(V_{\xi},V_{\eta};V_{\xi'}\oplus V_{\eta'})
\end{equation}
sending $(V^1_{\eta'}\subseteq V_{\eta'},V^1_{\xi'}\subseteq
V_{\xi'})$ to $(V^1_{\xi'}\oplus V^1_{\eta'}\subseteq
V_{\xi'}\oplus V_{\eta'})$ in a natural way. We set
$$
\overline{\mv}(V_{\xi},V_{\eta};V_{\xi'}\oplus
V_{\eta'}):=V(V_{\xi},V_{\eta};V_{\xi'}\oplus V_{\eta'})\backslash
\mbox{Im}i
$$
i.e.
\begin{equation}\label{decomposition}
\mv(V_{\xi},V_{\eta};V_{\xi'}\oplus
V_{\eta'})=\overline{\mv}(V_{\xi},V_{\eta};V_{\xi'}\oplus
V_{\eta'})\bigcup \mv_1
\end{equation}
where $\mv_1=\mbox{Im}i.$ Define
$$\mv_1(\delta,\beta):=\mbox{Im}(\mv(V_{\alpha},V_{\beta};
V_{\eta'})\times \mv(V_{\gamma},V_{\delta};V_{\xi'}))$$
\bigskip

Consider the $\bbc$-space
$M_G(A)=\bigoplus_{\ud\in\bbn^n}M_{G_{\ud}}(Q)$ where
$M_{G_{\ud}}(Q)$ is the $\bbc$-space of $G_{\ud}$-invariant
constructible function on $\bbe_{\ud}(Q).$ Define the convolution
multiplication on $M_G(A)$ by
$$f\bullet
g(L)=\sum_{c,d\in\bbc}\chi(\mv(f^{-1}(c),g^{-1}(d);L))cd.$$ for any
$f\in M_{G_{\ud}}(Q),g\in M_{G_{\ud'}}(Q)$ and $L\in
\bbe_{\ud+\ud'}.$

 As usual
for an algebraic variety $V$ and a constructible function $f$ on
$V,$  using the notation \eqref{integral} in Section \ref{basic},
we have
$$
f\bullet g(L)=\int_{\mv(supp(f),supp(g);L)}f(x')g(x'').
$$
The following is well-known (see \cite{Lusztig},\cite{Riedtmann}),
see a proof in \cite{DXX}.

\begin{Prop} The space $M_G(\Lambda)$ under the convolution
multiplication $\bullet$ is an associative $\bbc$-algebra with
unit element.
\end{Prop}

The above proposition implies the following identity
\begin{Thm}
For fixed $A$-modules $X,Y,Z$ and $M$ with dimension vectors
$\ud_X,\ud_Y,\ud_Z$ and $\ud_M$ such that
$\ud_M=\ud_X+\ud_Y+\ud_Z,$ we have
$$
\int_{\overline{L}\in \bbe_{\ud_X+\ud_Y}(A)/G_{\ud_X+\ud_Y}
}g_{XY}^{L}g_{LZ}^{M}=\int_{\overline{L'}\in
\bbe_{\ud_Y+\ud_Z}(A)/G_{\ud_Y+\ud_Z}}g_{XL'}^{M}g_{YZ}^{L'}
$$
\end{Thm}
Define
$$
\hspace{-9cm}W(X,Y;L_1,L_2):=$$$$\{(f,g,h)\mid
\xymatrix{0\ar[r]&Y\ar[r]^{f}&L_1\ar[r]^{g}&L_2\ar[r]^{h}&
X\ar[r]& 0}~~\mbox{is an exact sequence}\} .$$ Under the action
of $G_{\underline{\alpha}}\times G_{\underline{\beta}},$ where
$\underline{\alpha}=\udim X$ and $\underline{\beta}=\udim Y,$ the
orbit space is denoted by $\mv(X,Y;L_1,L_2).$ In fact,
$$
\mv(X,Y;L_1,L_2)=\{g:L_1\rightarrow L_2\mid \mathrm{Ker}g\cong Y
\mbox{ and } \mathrm{Coker}g\cong X\}
$$
Put
$$
h_{XY}^{L_1L_2}=\chi(\mv(X,Y;L_1L_2))
$$
We have the following ``higher order'' associativity.
\begin{Thm}\label{associativity}
For fixed $A$-modules $X,Y_i,L_i$ for $i=1,2,$ we have
$$
\int_{\overline{Y}}g_{Y_2Y_1}^{Y}h_{XY}^{L_1L_2}=\int_{\overline{L'_1}}
g_{L'_1Y_1}^{L_1}h_{XY_2}^{L'_1L_2}.
$$
Dually, for fixed $A$-modules $X_i,Y,L_i$ for $i=1,2,$ we have
$$
\int_{\overline{X}}g_{X_2X_1}^{X}h_{XY}^{L_1L_2}=\int_{\overline{L'_2}}
g_{X_2L'_2}^{L_2}h_{X_1Y}^{L_1L'_2}.
$$
\end{Thm}
\begin{proof}
Define
$$EF(X,Y_1,Y_2;L_1,L_2)=\{(g,Y^{\bullet})\mid g: L_1\rightarrow L_2, Y^{\bullet}=(\mathrm{Ker}
g\supseteq Y'\supseteq 0)$$$$\mbox{ such that }\mathrm{Coker}
g\simeq X,Y'\simeq Y_1, \mathrm{Ker}
g/Y'\simeq Y_2\}$$
and
$$EF'(X,Y_1,Y_2;L'_1,L_2)=\{(g',L_1^{\bullet})\mid g': L'_1\rightarrow L_2, L^{\bullet}=(L_1\supseteq Y'\supseteq 0)$$$$\mbox{ such that }\mathrm{Ker}
g'\simeq Y_2,\mathrm{Coker}
g'\simeq X,Y'\simeq Y_1, L_1/Y'\simeq L'_1\}$$ Consider the following diagram
\begin{equation}
\xymatrix{&Y_1\ar[d]\ar@{=}[r]&Y_1\ar[d]& &&\\
0\ar[r]& \mathrm{Ker}
g\ar[r]\ar[d]& L_1\ar[r]^{g}\ar[d]&
L_2\ar[r]\ar@{=}[d]&
X\ar[r]\ar@{=}[d]& 0\\
0\ar[r]& Y_2\ar[r]& L'_1\ar[r]^{g'}& L_2\ar[r] & X\ar[r]& 0}
\end{equation}
where $L'_1=L_1/Y'$ is the pushout. This gives the following
morphism of varieties:
$$
EF(X,Y_1,Y_2;L_1,L_2)\rightarrow EF'(X,Y_1,Y_2;L'_1,L_2)
$$
sending $(g,Y^{\bullet})$ to $(g',L_1^{\bullet})$ where
$g':L_1/Y'\rightarrow L_2.$ Conversely, we also have the morphism
$$
EF'(X,Y_1,Y_2;L'_1,L_2)\rightarrow EF(X,Y_1,Y_2;L_1,L_2)
$$
sending $(g',L_1^{\bullet})$ to $(g,Y^{\bullet})$ where $g$ is the
composition: $L_1\rightarrow L_1/Y'\simeq L'_1\xrightarrow{g'}L_2$
(this implies $Y'\subseteq Kerg$). A simple check shows there
%
exists a homeomorphism between $EF(X,Y_1,Y_2;L_1,L_2)$ and
$EF'(X,Y_1,Y_2;L'_1,L_2)$. By Proposition \ref{Euler2}, we have
$$
\chi(EF(X,Y_1,Y_2;L_1,L_2))=\int_{\overline{Y}}g_{Y_2Y_1}^{Y}h_{XY}^{L_1L_2}
$$
and
$$
\chi(EF'(X,Y_1,Y_2;L'_1,L_2))=\int_{\overline{L'_1}}
g_{L'_1Y_1}^{L_1}h_{XY_2}^{L'_1L_2}
$$
This completes the proof.
\end{proof}
We define $$\hom(L_1,L_2)_{Y[1]\oplus X}=\{g\in\hom(L_1,L_2)|\ker
g\simeq Y, \cok g\simeq X\}$$ Then, it is easy to identify that
$$ \mv(X,Y;L_1,L_2)=\hom(L_1,L_2)_{Y[1]\oplus X}$$ We can consider a $\bbc^{*}$-action on
 $\mathrm{Hom}
(L_1,L_2)_{Y[1]\oplus X}$
or $\mv(X,Y;L_1,L_2)$ simply by $t.(f,g,h)^{*}=(f,tg,h)^{*}$ for
$t\in \bbc^{*}$ and $(f,g,h)^{*}\in \mv(X,Y;L_1,L_2).$ We also have
a projective version of Theorem \ref{associativity}, where
$\mathbb{P}$ indicates the corresponding orbit space under the
$\bbc^*$-action.
\begin{Thm}\label{associativity2}
For fixed $A$-modules $X,Y_i,L_i$ for $i=1,2,$ we have
$$
\int_{\overline{Y}}g_{Y_2Y_1}^{Y}\chi(\mathbb{P}\mathrm{Hom}(L_1,L_2)_{Y[1]\oplus
X})=\int_{\overline{L'}_1}g_{L'_1Y_1}^{L_1}\chi(\mathbb{P}\mathrm{Hom}(L'_1,L_2)_{Y_2[1]\oplus
X}).
$$ Dually, for fixed $A$-modules $X_i,Y,L_i$ for $i=1,2,$ we have
$$
\int_{\overline{X}}g_{X_2X_1}^{X}\chi(\mathbb{P}\mathrm{Hom}(L_1,L_2)_{Y[1]\oplus
X})=\int_{\overline{L'}_2}g_{X_2L'_2}^{L_2}\chi(\mathbb{P}\mathrm{Hom}(L_1,L'_2)_{Y[1]\oplus
X_1}).
$$
\end{Thm}

\subsection{}
For fixed $\xi,\eta$ and $\xi',\eta'$ with
$\underline{\xi}+\underline{\eta}=\underline{\xi'}+\underline{\eta'}=\underline{\lz}$,
let $V_{\lz}\in\bbe_{\underline{\lz}}$ and $Q(V_{\lz})$ be the set
of $(a,b,a',b')$ such that the row and the column of the following
diagram are exact:
\begin{equation}\label{E:crosses}
\xymatrix{
& & 0 \ar[d] & &\\
& & V_{\eta} \ar[d]^-{a'} & & \\
0 \ar[r] & V_{\eta'} \ar[r]^-{a} & V_{\lambda} \ar[r]^{b} \ar[d]^-{b'} & V_{\xi'} \ar[r] & 0\\
& & V_{\xi}\ar[d] & &\\
& & 0 & &}
\end{equation}
We let
$$
Q(\xi, \eta, \xi', \eta')=\bigcup_{V_{\lambda}\in
\bbe_{\underline{\lambda}}}Q(V_{\lambda})
$$
We remark that $Q(\xi, \eta, \xi', \eta')$ can be viewed as a
constructible subset of the module variety
$\bbe_{(\underline{\xi},\underline{\eta},\underline{\xi'},\underline{\eta'},\underline{\lambda})}$
with
$\underline{\xi}+\underline{\eta}=\underline{\xi'}+\underline{\eta'}=\underline{\lambda}$
of the following quiver
\begin{equation}
\xymatrix{
 & 2 \ar[d] & \\
4 \ar[r]& 5 \ar[r] \ar[d] & 3  \\
 & 1 & }
\end{equation}

We have the following action of $G_{\underline{\lambda}}$ on
$Q(\xi, \eta,
\xi', \eta')$:
$$g.(a,b,a',b')=(ga,bg^{-1},ga',b'g^{-1}).$$
The orbit space of $Q(\xi, \eta, \xi', \eta')$ is denoted by
$Q(\xi, \eta, \xi', \eta')^{*}$ and the orbit of $(a,b,a',b')$ in
$Q(\xi, \eta, \xi',
\eta')^{*}$ is denoted by $(a,b,a',b')^{*}.$ We also have the
following action
of $G_{\underline{\lambda}}$ on
$W(V_{\xi'},V_{\eta'};\bbe_{\underline{\lambda}})$:
$g.(a,b)=(ga,bg^{-1}).$ In the induced orbit space
$\ext^1(V_{\xi'},V_{\eta'})$, the orbit of $(a,b)$ is denoted by
$(a,b)^{*}$.
Hence, we have
\begin{equation}\label{2}
\xymatrix{ W(V_{\xi'},V_{\eta'};\bbe_{\underline{\lambda}})\times
W(\xi,\eta;\bbe_{\underline{\lambda}})=Q(\xi, \eta, \xi',
\eta')\ar[d]_-{\phi_1}\ar[r]^-{\phi_2}& \mathrm{Ext}
^1(V_{\xi'},V_{\eta'})
\\ Q(\xi, \eta,
\xi', \eta')^{*}\ar[ur]_-{\phi}&}
\end{equation}
where $\phi((a,b,a',b')^{*})=(a,b)^{*}$ is well defined.

Let $(a,b,a',b')\in Q(V_{\lambda})$. We claim that the stabilizer
of $(a,b,a',b')$ in $\phi_1$ is $$a'e_1\hom(\cok b'a, \ker
ba')e_4b',$$ which is isomorphic to $\hom(\cok b'a, \ker ba'),$
where the injection $e_1: \ker ba'\rightarrow V_{\lambda}$ is
induced naturally by $a'$ and the surjection $e_4:
V_{\lambda}\rightarrow \cok b'a $ is induced naturally by $b'$. In
fact, consider the action of $G_{\underline{\lambda}}$ on
$W(V_{\xi},V_{\eta};\bbe_{\underline{\lambda}})$ given by
$g.(a',b')=(ga,bg^{-1}),$ the stabilizer of $(a',b')$ is
$1+a'\hom(V_{\xi},V_{\eta})b'$ (\cite{RingelGreen}). It is clear
that the stabilizer of  $(a,b,a',b')$ under the action given by
$\phi_1$ is the following subgroup:
$$
\{1+a'fb'\mid f\in \hom(V_{\xi},V_{\eta}), ba'fb'=0, a'fb'a=0\}.
$$
Since $b'$ is surjective and $a'$ is injective, $ba'fb'=0,
a'fb'a=0$ imply $ba'f=0, fb'a=0.$ This means $\Im f\in \ker ba'$
and $f(\ker b'a)=0.$ We easily deduce the above claim by this
conclusion. In the same way, the stabilizer of $(a,b)$ under the
action given by $\phi_2$ is $1+a\hom(V_{\xi'},V_{\eta'})b$, which
is isomorphic to
$\hom(V_{\xi'},V_{\eta'})$ just for $a$ is injective and $b$ is
surjective. We now compute the fibre
$\phi_1(\phi_2^{-1}((a,b)^{*}))$ of $\phi$ over $(a,b)^{*}.$
$$
\phi_2^{-1}((a,b)^{*})=(ga,bg^{-1},a',b')
$$
where $(a',b')\in W(V_{\xi},V_{\eta};V_{\lambda}).$ Fixed $a,b,$
Fix $a,b,$ and let $U=\{(a,b,a',b')\}$. Then
$$U\subset \phi_2^{-1}((a,b)^{*}), \quad
\phi_1(U)=\phi_1(\phi_2^{-1}((a,b)^{*}))$$
$\phi_1\mid_U:U\rightarrow \phi_1(U)$ can be viewed as the action
of the group $a\hom(V_{\xi'},V_{\eta'})b$ with the stable subgroup
%
$a'e_1\hom(\cok b'a, \ker ba')e_4b',$ i.e., the fibre of
$\phi_1\mid_U$ is isomorphic to
$a\hom(V_{\xi'},V_{\eta'})b/a'e_1\hom(\cok b'a, \ker ba')e_4b'.$
Hence, by Corollary \ref{euler3},
$$\chi(W(V_{\xi},V_{\eta};V_{\lambda}))=\chi(U)=\chi(\phi_1(U))=\chi(\phi_1(\phi_2^{-1}((a,b)^{*}))).$$
Moreover, consider the action of $G_{\underline{\xi}}\times
G_{\underline{\eta}}$ on $Q(\xi,\eta,\xi',\eta')^{*}$ and the
induced orbit space, denoted by $Q(\xi, \eta, \xi',
\eta')^{\wedge}.$  The stabilizer
$\operatorname{Stab}_G((a,b,a',b')^{*})$ of
$(a,b,a',b')^{*}$ is
$$
\{(g_1,g_2)\in G_{\underline{\xi}}\times G_{\underline{\eta}}\mid
ga'=a'g_2,b'g=g_1b'\quad \mbox{for some }g\in
1+a\hom(V_{\xi'},V_{\eta'})b\},
$$
also denoted by $G((a,b,a',b')^{*})$. This determines the group
embedding
$$\operatorname{Stab}_G((a,b,a',b')^{*})\ra
(1+a\hom(V_{\xi'},V_{\eta'})b)/(1+ae_1\hom(\cok b'a, \ker
ba')e_4b).$$  The group $G((a,b,a',b')^{*})$ is isomorphic to a
vector space since $ba=0.$ We know that
$1+a\hom(V_{\xi'},V_{\eta'})b$ is the subgroup of $\aut V_{\lz},$ it
acts on $W(V_{\xi},V_{\eta};\bbe_{\underline{\lambda}})$ naturally.
 The orbit space
of $W(V_{\xi},V_{\eta};\bbe_{\underline{\lambda}})$ under the
action of $1+a\mathrm{Hom}(V_{\xi'},V_{\eta'})b$ is denoted by
$\widetilde{W}(V_{\xi},V_{\eta};\bbe_{\underline{\lambda}})$ and
similar considerations hold for
$\mv(V_{\xi},V_{\eta};\bbe_{\underline{\lambda}}).$
 Combined with the
discussion above, we have the following commutative diagram of
actions of groups:

\begin{equation}\label{stablegroup}
\xymatrix{W(V_{\xi},V_{\eta};\bbe_{\underline{\lambda}})\ar[rrr]^{1+a\hom(V_{\xi'},V_{\eta'})b}
\ar[d]^{G_{\underline{\xi}}\times G_{\underline{\eta}}}&&&
\widetilde{W}(V_{\xi},V_{\eta};\bbe_{\underline{\lambda}})\ar[d]^{G_{\underline{\xi}}\times
G_{\underline{\eta}}}\\
\mv(V_{\xi},V_{\eta};\bbe_{\underline{\lambda}})\ar[rrr]^{1+a\hom(V_{\xi'},V_{\eta'})b}&&&\widetilde{\mv}(V_{\xi},V_{\eta};\bbe_{\underline{\lambda}})}
\end{equation}
The stabilizer of $(a', b')^{\wedge}$ in the bottom map is
%
$$\{g\in 1+a\hom(V_{\xi'},V_{\eta'})b\mid ga'=a'g_2, b'g=g_1b' \mbox{ for some }(g_1,g_2)\in G_{\underline{\xi}}\times G_{\underline{\eta}}\}$$
which is isomorphic to a vector space too, it is denoted by
$V(a,b,a',b').$ We can construct the map from $V(a,b,a',b')$ to
$\operatorname{Stab}_G((a,b,a',b')^{*})$
sending $g$ to $(g_1,g_2).$ It is well-defined
since $a'$ is injective and $b'$ is surjective. We have
$$
V(a,b,a',b')/\hom(\cok b'a, \ker ba')\cong
\operatorname{Stab}_G((a,b,a',b')^{*}).
$$

We have the following proposition.

\begin{Prop}\label{fibre1} The fiber over $(a,b)^{*}\in\ext^1(V_{\xi'},V_{\eta'})_{\lambda}$ of the surjective map
$$\phi^{\wedge}: Q(\xi,
\eta, \xi', \eta')^{\wedge}\rightarrow \mathrm{Ext}
^1(V_{\xi'},V_{\eta'})$$ is isomorphic to
$\widetilde{\mv}(V_{\xi},V_{\eta};\bbe_{\lambda})$, where
$\widetilde{\mv}(V_{\xi},V_{\eta};\bbe_{\lambda})$ is such that
there exists a surjective morphism from
$\mv(V_{\xi},V_{\eta};V_{\lambda})$ to
$\widetilde{\mv}(V_{\xi},V_{\eta};\bbe_{\lambda})$ such that any
fibre is isomorphic to an affine space of dimension
$$\mathrm{dim}_{\bbc}\mathrm{Hom}(V_{\xi'},V_{\eta'})-\mathrm{dim}_{\bbc}\mathrm{Hom}(\mathrm{Coker} b'a,\mathrm{Ker}
ba')- \mathrm{dim}_{\bbc}\mathrm{Stab}_{G}((a,b,a',b')^{*})).$$
\end{Prop}

Also, we also have a commutative diagram induced by \eqref{2}. By
Proposition \ref{Euler2}, we have
\begin{Cor} The following equality holds.
\begin{equation}
\sum_{\lambda}\chi(Q(\lambda)^{\wedge})=\sum_{\lambda}g^{\lambda}_{\beta\alpha}h_{\lambda}^{\xi'\eta'}
\end{equation}
\end{Cor}

\subsection{}
Let $\mo(\xi, \eta, \xi', \eta')$ be the set of
$(V_{\delta},V_{\beta},e_1,e_2,e_3,e_4,c,d)$ such that the
following commutative diagram has exact rows and columns:
\begin{equation}\label{guru}
\xymatrix{ & 0 \ar[d] &&&&  0\ar[d] && \\
0 \ar[r] & V_{\beta} \ar[rr]^-{e_1} \ar[dd]^-{u'} && V_{\eta}\ar[rr]^-{e_2} \ar@{.>}[dl]^-{u_{V_{\eta}}}\ar@{.>}[dd]&& V_{\delta} \ar[r] \ar[dd]^-{x}& 0\\
&&S\ar@{.>}[dr]^-{d}&&&&\\
& V_{\eta'} \ar@{.>}[ur]_-{u_{V_{\eta'}}}\ar[dd]^-{v'} \ar@{.>}[rr]&& V_{\lambda}\ar@{.>}[rr]\ar@{.>}[dd]\ar@{.>}[dr]^-{c}&& V_{\xi'} \ar[dd]^-{y} &\\
&&&&T\ar@{.>}[ur]_-{q_{V_{\xi'}}}\ar@{.>}[dl]^-{q_{V_{\xi}}}&&\\
0\ar[r] &  V_{\alpha} \ar[d] \ar[rr]^-{e_3} && V_{\xi} \ar[rr]^-{e_4} && V_{\gamma} \ar[r] \ar[d] & 0\\
& 0 && && 0 &}
\end{equation}
where $V_{\delta},V_{\beta}$ are submodules of
$V_{\xi'},V_{\eta'},$ respectively;
$V_{\gamma}=V_{\xi'}/V_{\delta},$  $
V_{\alpha}=V_{\eta'}/V_{\beta},$ $u', x, v', y$ are the canonical
morphisms, and $V_{\lambda}$ is the center induced by the above
square, $T=V_{\xi} \times_{V_{\gamma}} V_{\xi'} =\{(x \oplus m)
\in V_{\xi} \oplus V_{\xi'}\;|\; e_4(x)=y(m)\}$ and $S=V_{\eta}
\sqcup_{\small{V_{\beta}}} \hspace{-.05in}V_{\eta'} =V_{\eta}
\oplus V_{\eta'} / \{e_1(v_{\beta})\oplus
u'(v_{\beta})\;|\;v_{\beta} \in V_{\beta}\}.$ Then there is unique
map $f:S\rightarrow T$ for the fixed square. Let $(c,d)$ be a pair
of maps such that $c$ is surjective, $d$ is injective and
$cd=f.$ In particular, for fixed submodules $V_{\delta}$ and
$V_{\beta}$ of $V_{\xi'}$ and $V_{\eta'}$ respectively, the subset
of $\mo(\xi,\eta,\xi',\eta')$
$$
\{(V_{1},V_{2},e_1,e_2,e_3,e_4,c,d)\in
\mo(\xi,\eta,\xi',\eta')\mid V_1=V_{\delta}, V_2=V_{\beta}\}
$$
is denoted by $\mo_{(V_{\gamma},V_{\delta},V_{\alpha},V_{\beta})}$
where $V_{\gamma}=V_{\xi'}/V_{\delta}$ and $
V_{\alpha}=V_{\eta'}/V_{\beta}$.
There is a natural action of the group $G_{\underline\lz}$ on
$\mo_{(V_{\gamma},V_{\delta},V_{\alpha},V_{\beta})}$ as follows:
$$
g.(V_{\delta},V_{\beta},e_1,e_2,e_3,e_4,c,d)=(V_{\delta},V_{\beta},e_1,e_2,e_3,e_4,cg^{-1},gd)
$$

We denote by
$\mo^{*}_{(V_{\gamma},V_{\delta},V_{\alpha},V_{\beta})}$ and
$\mo(\xi, \eta, \xi', \eta')^{*}$ the orbit spaces under the
actions of $G_{\underline\lz}.$

\subsection{}
There is a homeomorphism between $Q(\xi, \eta, \xi', \eta')^{*}$
and $\mo(\xi, \eta, \xi', \eta')^{*}$ ( see \cite{DXX}):
$$
\theta^{*}:Q(\xi, \eta, \xi', \eta')^{*}\rightarrow \mo(\xi, \eta,
\xi', \eta')^{*}
$$
induced by the map between $Q(\xi, \eta, \xi', \eta')$ and
$\mo(\xi, \eta, \xi', \eta')$ defined as follows:
$$V_{\beta}=\text{Ker}\; b'a \simeq \text{Ker}\; ba', \qquad V_{\delta}=\text{Im}\;ba', $$
$$e_1=(a')^{-1}a, \quad e_2=ba', \quad e_3=b'a, \quad e_4=b(b')^{-1}$$
and $c,d$ are induced by the maps:
$$
V_{\eta}\oplus V_{\eta'} \rightarrow V_{\lambda} \quad\mbox{ and }
\quad V_{\lambda}\rightarrow V_{\xi}\oplus V_{\xi'}
$$
There is an action of $G_{\underline{\xi}}\times
G_{\underline{\eta}}$ on $\mo(\xi,\eta,\xi',\eta')^{*},$ defined
as follows: for $(g_1,g_2)\in G_{\xi}\times G_{\eta},$
$$
(g_1,g_2).(V_{\delta},V_{\beta},e_1,e_2,e_3,e_4,c,d)^{*}=(V_{\delta},V_{\beta},g_2e_1,e_2g_2^{-1},g_1e_3,e_4g_1^{-1},c',d')^{*}
$$
Let us determine the relation between $(c',d')$ and $(c,d).$

Suppose that
$(V_{\delta},V_{\beta},g_2e_1,e_2g_2^{-1},g_1e_3,e_4g_1^{-1})$
induces $S',T'$ and the unique map $f': S'\rightarrow T',$ then it
is clear that there are isomorphisms:
$$a_1: S\rightarrow S' \quad \mbox{and} \quad a_2: T\rightarrow T'$$
induced by isomorphisms:
$$
\left(%
\begin{array}{cc}
  g_2 & 0\\
  0& id \\
\end{array}%
\right): V_{\eta}\oplus V_{\eta'}\rightarrow V_{\eta}\oplus V_{\eta'} \quad \mbox{and} \quad \left(%
\begin{array}{cc}
  g_1 & 0\\
  0& id \\
\end{array}%
\right): V_{\xi}\oplus V_{\xi'}\rightarrow V_{\xi}\oplus V_{\xi'}
$$

 So $f'=a_2fa_{1}^{-1},$ we have the
following commutative diagram:
\begin{equation}\label{cd}
\xymatrix{S'\ar[r]^{d'}& V_{\lambda}\ar[r]^{d'_1}& V_{\gamma}\ar@{=}[d]\\
S\ar[r]^{d}\ar[u]^{a_1}\ar[d]& V_{\lambda}\ar[r]^{d_1}\ar[d]^{c}\ar[u]^{g}& V_{\gamma}\ar@{=}[d]\\
\Im f\ar[r]\ar[d]^{a_2}& T\ar[r]\ar[d]^{a_2}& V_{\gamma}\ar@{=}[d]\\
\Im f'\ar[r]& T'\ar[r]& V_{\gamma}}
\end{equation}
Hence, $c'=a_2cg^{-1}$ and $d'=gda_{1}^{-1}.$ In particular,
$c=c'$ and $d=d'$ if and only if $g_1=id_{V_{\xi}}$ and
$g_2=id_{V_{\eta}}.$ This shows that the action of
$G_{\underline{\xi}}\times G_{\underline{\eta}}$ is free.

Its orbit space is denoted by $\mo(\xi,\eta,\xi',\eta')^{\wedge}.$
The homeomorphism $\theta^{*}$ above induces the homeomorphism in
the following the Proposition:
\begin{Prop}\label{fibre2}
There exists a homeomorphism under quotient topology
$$
\theta^{\wedge}:Q(\xi,\eta,\xi',\eta')^{\wedge}\rightarrow
\mo(\xi,\eta,\xi',\eta')^{\wedge}.
$$
\end{Prop}

Let $\md(\xi,\eta,\xi',\eta')^{*}$ be the set of
$(V_{\delta},V_{\beta},e_1,e_2,e_3,e_4)$ such that the diagram
\eqref{guru} is commutative and has exact rows and columns.
In particular, for fixed $V_{\delta}$ and $V_{\beta}$, its subset
$$
\{(V_{1},V_{2},e_1,e_2,e_3,e_4)\mid V_1=V_{\delta},
V_2=V_{\beta}\}
$$
is denoted by
$D^{*}_{(V_{\gamma},V_{\delta},V_{\alpha},V_{\beta})}$ where
$V_{\gamma}=V_{\xi'}/V_{\delta}$ and $
V_{\alpha}=V_{\eta'}/V_{\beta}$.
Then we
have a projection:
$$
\varphi^{*}:\mo(\xi,\eta,\xi',\eta')^{*}\rightarrow
\md(\xi,\eta,\xi',\eta')^{*}
$$

We claim that the fibre of this morphism is isomorphic to a vector
space which has the same dimension as
$\mathrm{Ext}^1(V_{\gamma},V_{\beta})$ for any
element in $\md_{(V_{\gamma},V_{\delta},V_{\alpha},V_{\beta})}^{*}.$

Fix an element $(V_{\delta},V_{\beta},e_1,e_2,e_3,e_4)\in
\md^{*}_{(V_{\gamma},V_{\delta},V_{\alpha},V_{\beta})},$ and let
$V$
be the set consisting of the equivalence classes $(c,d)^*$of
elements $(c,d)$ under the action of $G_{\underline{\lambda}}$
such that the following diagram is commutative:
\begin{equation}\label{square1}
\xymatrix{V_{\beta}\ar@{=}[r]\ar[d]^-{u_1}& V_{\beta}\ar[r]\ar[d]^-{s}& 0\ar[d]\\
S\ar[r]^-{d}\ar[d]^-{v_1}& V_{\lambda}\ar[r]^-{t}\ar[d]^-{c}& V_{\gamma}\ar@{=}[d]\\
L\ar[r]^-{u_2}& T\ar[r]^-{v_2}& V_{\gamma}}
\end{equation}
where $u_1,u_2,v_1,v_2$ are fixed and come from the long exact
sequence:
\begin{equation}
\xymatrix{&&&L=\mathrm{Im}f
\ar[dr]^-{u_2}&&&\\
0\ar[r]&V_{\beta}\ar[r]^-{u_1}&
S\ar[ur]^-{v_1}\ar[rr]^-{f}&&T\ar[r]^-{v_2}&V_{\gamma}\ar[r]&0}
\end{equation}
and where $s,t$ are naturally induced by $c,d$ respectively. Then
$V$ is the
fibre of $\varphi^*.$ We note that $c$ is in diagram
\eqref{square1} if and only if $c\in u_2\hom(V_{\gamma},L)t.$ Of
course, $u_2\hom(V_{\gamma},L)t$ is isomorphic to
$\hom(V_{\gamma},L).$

Let $\varepsilon_0\in \ext^1(V_{\gamma},L)$ be the class of the
following exact sequence:
$$
\xymatrix{0\ar[r]&L\ar[r]^-{u_2}&T\ar[r]^-{v_2}&V_{\gamma}
\ar[r]&0}
$$
The above long exact sequence induces the following long exact
sequence:
\begin{equation}\label{E:Cross4}
\xymatrix{ 0 \ar[r] & \text{Hom}(V_{\gamma},V_{\beta}) \ar[r] &
\text{Hom}(V_{\gamma},S) \ar[r]& \text{Hom}(V_{\gamma},L) \ar[r]
&}
\end{equation}
$$\qquad\xymatrix{
 \ar[r] & \text{Ext}^1(V_{\gamma},V_{\beta}) \ar[r] & \text{Ext}^1(V_{\gamma},S) \ar[r]^-{\phi} & \text{Ext}^1(V_{\gamma},L) \ar[r] &
 0}.
$$
Consider the morphism $\omega:V\rightarrow
\phi^{-1}(\varepsilon_0)$ sending $(c,d)^{*}$ to $(d,t)^{*}.$
$$
\omega^{-1}((d,t)^{*})=\{(cg^{-1},gd)^{*}\mid g\in
G_{\underline{\lambda}}\}=\{(cg^{-1},d)^{*}\mid g\in
1+d\hom(V_{\gamma},S)t \}
$$
Hence, the fibre of $\omega$ can be viewed as the orbit space of
$u_2\hom(V_{\gamma},L)t$ under the action of
$1+d\hom(V_{\gamma},S)t$ given by $g.u_2ft=u_2fg_{\gamma}^{-1}t$
where
$g\in 1+d\hom(V_{\gamma},S)t$ and $g_{\gamma}$ is the isomorphism
on $V_{\gamma}$ induced by $g.$ with the stabilizer
isomorphic to the vector space $\hom(V_{\gamma},V_{\beta}).$
Hence, up
to a translation from $\varepsilon_0$ to $0,$ $V$ is isomorphic to
the affine space
$$
\phi^{-1}(\varepsilon_0)\times \hom(V_{\gamma},L)\times
\hom(V_{\gamma},V_{\beta})/\hom(V_{\gamma},S)
$$
which is denoted by $W(V_{\delta},V_{\beta},e_1,e_2,e_3,e_4),$ and
whose dimension is $\dim_{\bbc}\ext^1(V_{\gamma},V_{\beta}).$

There is also an action of the group $G_{\underline{\xi}}\times
G_{\underline{\eta}}$ on $\md(\xi,\eta,\xi',\eta')^{*}$ with
stabilizer isomorphic to the vector space
$\hom(V_{\gamma},V_{\alpha})\times \hom(V_{\delta},V_{\beta}).$
The orbit space is denoted by $\md(\xi,\eta,\xi',\eta')^{\wedge}.$
The projection $\varphi^*$ naturally induces the projection:
$$
\varphi^{\wedge}: \mo(\xi,\eta,\xi',\eta')^{\wedge}\rightarrow
\md(\xi,\eta,\xi',\eta')^{\wedge}
$$
Its fibre over $(V_{\delta},V_{\beta},e_1,e_2,e_3,e_4)^{\wedge}$
is
isomorphic to the quotient space of
$$(\varphi^{*})^{-1}(V_{\delta},V_{\beta},e_1,e_2,e_3,e_4)$$ under the action of
$\hom(V_{\gamma},V_{\alpha})\times \hom(V_{\delta},V_{\beta}).$
The corresponding stabilizer of
$(V_{\delta},V_{\beta},e_1,e_2,e_3,e_4,c,d)^{*}\in
(\varphi^{*})^{-1}(V_{\delta},V_{\beta},e_1,e_2,e_3,e_4)$ is
$$
\{(g_1,g_2)\in 1+e_3\mathrm{Hom}(V_{\gamma},V_{\alpha})e_4\times
1+e_1\mathrm{Hom}(V_{\delta},V_{\beta})e_2\mid
$$$$ga'=a'g_2,b'g=g_1b'\quad \mbox{for some }g\in 1+a\mathrm{Hom}(V_{\xi'},V_{\eta'})b\}
$$
where $(a,b,a',b')^{*}$ is induced by $\theta^*$ as showed in
diagram \eqref{guru}. It is isomorphic to the vector space
$\operatorname{Stab}_G((a,b,a',b')^{*}).$
Therefore, we have
\begin{Prop}\label{fibre3}
There exists a projection $$ \varphi^{\wedge}:
\mo(\xi,\eta,\xi',\eta')^{\wedge}\rightarrow
\md(\xi,\eta,\xi',\eta')^{\wedge}
$$
such that any fibre for $(V_{\delta}, V_{\beta}, e_1, e_2, e_3,
e_4)^{\wedge}$ is isomorphic to an affine space of dimension
$$
\mathrm{dim}_{\bbc}\mathrm{Ext}^1(V_{\gamma},V_{\beta})+\mathrm{dim}_{\bbc}
\mathrm{Stab}_{G}((a,b,a',b')^{*})-\mathrm{dim}_{\bbc}\mathrm{Hom}(V_{\gamma},V_{\alpha})-\mathrm{dim}_{\bbc}
\mathrm{Hom}(V_{\delta},V_{\beta})
$$
where $V_{\gamma}\simeq V_{\xi'}/V_{\delta}$ and $V_{\alpha}\simeq
V_{\eta'}/V_{\beta}$.
\end{Prop}

    Let us summarize the discussion above in the following diagram:
\begin{equation}\label{picture}
\xymatrix{\mathrm{Ext}
^1(V_{\xi'},V_{\eta'})&
Q(\xi,\eta,\xi',\eta')^{\wedge}\ar[l]_-{\phi^{\wedge}}\ar[r]^-{\theta^{\wedge}}&\mo(\xi,\eta,\xi',\eta')^{\wedge}\ar[r]^-{\varphi^{\wedge}}&
\md(\xi,\eta,\xi',\eta')^{\wedge}}.
\end{equation}

The following theorem can be viewed as a degenerated version of
Green's formula.
\begin{Thm}\label{degenerated}
For fixed $\xi, \eta, \xi', \eta',$ we have
$$
g_{\xi\eta}^{\xi'\oplus
\eta'}=\int_{\alpha,\beta,\delta,\gamma;\alpha\oplus
\gamma=\xi,\beta\oplus \delta=
\eta}g_{\gamma\delta}^{\xi'}g_{\alpha \beta}^{\eta'}
$$
\end{Thm}
\begin{proof}
We note that
$$
\chi(\md(\xi,\eta,\xi',\eta')^{\wedge})=\int_{\alpha,\beta,\delta,\gamma}g_{\gamma\delta}^{\xi'}g_{\alpha\beta}^{\eta'}h^{\gamma\alpha}_{\xi}h^{\delta\beta}_{\eta}
$$
Because the Euler characteristic of an affine space is
$1,$ we have
$$
\int_{\lambda}h^{\xi'\eta'}_{\lambda}g_{\xi\eta}^{\lambda}=\int_{\alpha,\beta,\delta,\gamma}g_{\gamma\delta}^{\xi'}g_{\alpha\beta}^{\eta'}h^{\gamma\alpha}_{\xi}h^{\delta\beta}_{\eta}
$$
Using Proposition \ref{Euler2} and Lemma \ref{directsum}, we
simplify the identity as
$$
h^{\xi'\eta'}_{\xi'\oplus \eta'}g_{\xi\eta}^{\xi'\oplus
\eta'}=\int_{\alpha,\beta,\delta,\gamma, \alpha\oplus \gamma=\xi,
\beta\oplus
\gamma=\eta}g_{\gamma\delta}^{\xi'}g_{\alpha\beta}^{\eta'}h^{\gamma\alpha}_{\xi}h^{\delta\beta}_{\eta}
$$
i.e.,
$$
g_{\xi\eta}^{\xi'\oplus
\eta'}=\int_{\beta,\delta}g_{\xi'/\delta,\delta}^{\xi'}g_{\eta'/\beta
,\beta}^{\eta'}.
$$
\end{proof}

\subsection{}
We define $EF(\xi,\eta,\xi',\eta')$ to be the set
$$
\{(\varepsilon, L'(d))\mid \varepsilon\in \mathrm{Ext}^1(V_{\xi'},
V_{\eta'})_{L(d)}, L'(d)\subseteq L(d), L'(d)\simeq V_{\eta},
L(d)/L'(d)\simeq V_{\xi}\}
$$
and let $FE(\xi,\eta,\xi',\eta')$ be the set
$$
\{(V'_{\xi'}, V'_{\eta'}, \varepsilon_1, \varepsilon_2)\mid
V'_{\xi'}\subseteq V_{\xi}, V'_{\eta'}\subseteq V_{\eta'},
\varepsilon_1\in \mathrm{Ext}^1(V'_{\xi'}, V'_{\eta'})_{V_{\eta}},
\varepsilon_2\in\mathrm{Ext}^1(V_{\xi'}/V'_{\xi'},
V_{\eta'}/V'_{\eta'})_{V_{\xi}} \}.
$$

The projection $$p_1: EF(\xi,\eta,\xi',\eta')\rightarrow
\mathrm{Ext}^1(V_{\xi'}, V_{\eta'})$$ satisfies that the fibre of
any $\varepsilon\in \mathrm{Ext}^1(V_{\xi'}, V_{\eta'})_{L(d)}$ is
isomorphic to $\mv(V_{\xi}, V_{\eta}; L(d))$.

Comparing with Proposition \ref{fibre1}, we have a morphism
$$
EF(\xi,\eta,\xi',\eta')\rightarrow Q(\xi,\eta,\xi',
\eta')^{\wedge}
$$
satisfying the fibre  of $(a, b, a', b')^{\wedge}$ is isomorphic
to an affine space of dimension
$$
\mathrm{dim}_{\bbc}\mathrm{Hom}(V_{\xi'},
V_{\eta'})-\mathrm{dim}_{\bbc}\mathrm{Hom}(\mathrm{Coker}b'a,
\mathrm{Ker}ba')-\mathrm{dim}_{\bbc}G((a,b,a',b')^{*}).
$$

We also have a natural homeomorphism
$$
FE(\xi,\eta,\xi',\eta')\rightarrow \md(\xi, \eta, \xi',
\eta')^{\wedge}.
$$
Hence, using Proposition \ref{fibre1}, \ref{fibre2} and
\ref{fibre3}, we have
\begin{Prop}\label{the map}
There is a natural morphism
$$
\rho: EF(\xi,\eta,\xi',\eta')\rightarrow FE(\xi,\eta,\xi',\eta')
$$
satisfying the fibre for $(V'_{\xi'}, V'_{\eta'}, \varepsilon_1,
\varepsilon_2)$ is isomorphic to an affine space of dimension
$$
\mathrm{dim}_{\bbc}\mathrm{Hom}(V_{\xi'},
V_{\eta'})-\mathrm{dim}_{\bbc}\mathrm{Hom}(V_{\gamma}, V_{\beta})+
\mathrm{dim}_{\bbc}\mathrm{Ext}(V_{\gamma},
V_{\beta})-\mathrm{dim}_{\bbc}\mathrm{Hom}(V_{\gamma},
V_{\alpha})-\mathrm{dim}_{\bbc}\mathrm{Hom}(V_{\delta},
V_{\beta}).
$$
where $V_{\beta}\simeq V'_{\eta'}, V_{\delta}\simeq V'_{\xi'}$ and
$V_{\alpha}\simeq V_{\eta'}/V'_{\eta'}, V_{\gamma}\simeq
V_{\xi'}/V'_{\xi'}.$
\end{Prop}

Now we consider the action of $\mathbb{C}^{*}$ on
$EF(\xi,\eta,\xi',\eta')$ and $FE(\xi,\eta,\xi',\eta')$.

\nd (1) For $t\in \bbc^{*}$ and $(\varepsilon, L'(d))\in EF(\xi,
\eta, \xi', \eta'),$ we know $t.\varepsilon\in
\mathrm{Ext}^1(V_{\xi'}, V_{\eta'})_{L(t.d)}$ and
$L(d)=V_{\xi'}\oplus V_{\eta'}$ as a direct sum of vector spaces.
Recall that $L(t.d)$ is defined in Section 4.1. Define
$$L'(t.d):=\{(v', tv'')\mid (v',v'')\in L'(d)\}.$$
Then $L'(t.d)\subseteq L(t.d).$ Hence, we define
$$
t.(\varepsilon, L'(d))=(t.\varepsilon, L'(t.d)).
$$
The orbit space is denoted by $\widehat{EF}(\xi,\eta,\xi',\eta')$.
A point $(\varepsilon, L'(d))$ is stable, i.e., $t.(\varepsilon,
L'(d))=(\varepsilon, L'(d))$ for any $t\in \bbc^{*}$ if and only
if
$$
L(d)=V_{\xi'}\oplus V_{\eta'} \mbox{ and }L'(d)=(L'(d)\cap
V_{\xi'})\oplus(L'(d)\cap V_{\eta'})
$$
Note that the above direct sums are the direct sums of modules.
The set of stable points in $EF(\xi,\eta,\xi',\eta')$ is denoted
by $EF_{s}(\xi,\eta,\xi',\eta').$ The action of $\bbc^{*}$ on the
set of non-stable points in $EF(\xi, \eta, \xi', \eta')$ is free.
We denote the orbit space by $\mathbb{P}EF(\xi, \eta, \xi',
\eta').$ Of course, we have
$$
\widehat{EF}(\xi,\eta,\xi',\eta')=EF_{s}(\xi,\eta,\xi',\eta')\bigcup
\mathbb{P}EF(\xi, \eta, \xi', \eta').
$$
The orbit of $(\varepsilon, L'(d))$ in  $\mathbb{P}EF(\xi, \eta,
\xi', \eta')$ is denoted by $\mathbb{P}(\varepsilon, L'(d))$.

\nd (2) For $t\in \bbc^{*}$ and $(V'_{\xi'}, V'_{\eta'},
\varepsilon_1, \varepsilon_2)\in FE(\xi, \eta, \xi', \eta'),$ we
define
$$
t.(V'_{\xi'}, V'_{\eta'}, \varepsilon_1,
\varepsilon_2)=(V'_{\xi'}, V'_{\eta'}, t.\varepsilon_1,
t.\varepsilon_2).
$$
The orbit space is denoted by $\widehat{FE}(\xi,\eta,\xi',\eta').$
A point $(V'_{\xi'}, V'_{\eta'}, \varepsilon_1, \varepsilon_2)$ in
$FE(\xi, \eta, \xi', \eta')$ is stable if and only if
$\varepsilon_1=\varepsilon_2=0.$ The set of stable points in
$FE(\xi, \eta, \xi', \eta')$ is denoted by $FE_{s}(\xi, \eta,
\xi', \eta').$ The action of $\bbc^{*}$ on the set of non-stable
points in $FE(\xi, \eta, \xi', \eta')$ is free. We denote the
orbit space by $\mathbb{P}FE(\xi, \eta, \xi', \eta').$ Of course,
we have
$$
\widehat{FE}(\xi,\eta,\xi',\eta')=FE_{s}(\xi,\eta,\xi',\eta')\bigcup
\mathbb{P}FE(\xi, \eta, \xi', \eta').
$$
The orbit of $(V'_{\xi'}, V'_{\eta'}, \varepsilon_1,
\varepsilon_2)$ in  $\mathbb{P}FE(\xi, \eta, \xi', \eta')$ is
denoted by $\mathbb{P}(V'_{\xi'}, V'_{\eta'}, \varepsilon_1,
\varepsilon_2)$.

The morphism $\rho$ induces the morphism
$$
\hat{\rho}: \widehat{EF}(\xi,\eta,\xi',\eta')\rightarrow
\widehat{FE}(\xi,\eta,\xi',\eta')
$$
We consider its restriction to
$\mathbb{P}EF(\xi,\eta,\xi',\eta').$
$$
\hat{\rho}\mid_{\mathbb{P}EF(\xi,\eta,\xi',\eta')}:
\mathbb{P}EF(\xi,\eta,\xi',\eta')\rightarrow
\widehat{FE}(\xi,\eta,\xi',\eta')
$$
For any $(V'_{\xi'}, V'_{\eta'}, 0, 0)\in
FE_{s}(\xi,\eta,\xi',\eta')\subseteq
\widehat{FE}(\xi,\eta,\xi',\eta'),$ we have
$$
(\hat{\rho}\mid_{\mathbb{P}EF(\xi,\eta,\xi',\eta')})^{-1}((V'_{\xi'},
V'_{\eta'}, 0, 0))=\mathbb{P}(\rho^{-1}(V'_{\xi'}, V'_{\eta'}, 0,
0)\setminus(V'_{\xi'}\oplus V'_{\eta'}, 0)).
$$
It is actually the projective space of the affine space in
Proposition \ref{the map}.  For any $\mathbb{P}(V'_{\xi'},
V'_{\eta'}, \varepsilon_1, \varepsilon_2)\in
\mathbb{P}FE(\xi,\eta,\xi',\eta'),$
$(\hat{\rho}\mid_{\mathbb{P}EF(\xi,\eta,\xi',\eta')})^{-1}(\mathbb{P}(V'_{\xi'},
V'_{\eta'}, \varepsilon_1, \varepsilon_2))$ is isomorphic to the
affine space in Proposition \ref{the map}. Now we  compute the
Euler characteristics. By Proposition \ref{Euler} and the above
discussion of the fibres, we have
$$
\hspace{-0cm}\chi(\mathbb{P}EF(\xi,\eta,\xi',\eta'))=$$$$\int_{(V'_{\xi'},
V'_{\eta'}, 0, 0)\in
FE_{s}(\xi,\eta,\xi',\eta')}[d(\xi',\eta')-d(\gamma,\alpha)-d(\delta,\beta)-\lr{\gamma,\beta}]g_{\gamma\delta}^{\xi'}g_{\alpha\beta}^{\eta'}
+\chi(\mathbb{P}FE(\xi,\eta,\xi',\eta')).
$$
where $d(\gamma,\alpha)
=\dim_{\mathbb{C}}\hom_{\Lambda}(V_{\gamma},V_{\alpha})$ and the
Euler form
$\lr{\gamma,\beta}=\dim_{\mathbb{C}}\hom(V_{\gamma},V_{\beta})-\dim_{\mathbb{C}}\ext^1(V_{\gamma},V_{\beta}).$
On the other hand, we know
$$\chi(\mathbb{P}EF(\xi,\eta,\xi',\eta'))=$$$$\int_{\alpha,\beta,\delta,\gamma,\alpha\oplus \gamma=\xi,
\beta\oplus\delta=
\eta}\chi(\mathbb{P}\overline{\mathcal{V}}(V_{\xi},V_{\eta};V_{\xi'}\oplus
   V_{\eta'}))+\int_{\lambda\neq \xi'\oplus
\eta'}\chi(\mathbb{P}Ext^{1}(V_{\xi'},V_{\eta'})_{\lambda})g_{\xi\eta}^{\lambda}.
$$ Therefore, we have the following theorem, which can be viewed as
a geometric version of Green's formula under the $\bbc^*$-action.
\begin{Thm}\label{projgreen}
For fixed $\xi, \eta, \xi', \eta',$ we have
\begin{eqnarray*}
 &&\int_{\lambda\neq \xi'\oplus
\eta'}\chi(\mathbb{P}\mathrm{Ext}^{1}(V_{\xi'},V_{\eta'})_{\lambda})g_{\xi\eta}^{\lambda}=\\
   &&\int_{\alpha,\beta,\delta,\gamma,\alpha\oplus \gamma\neq \xi
\mbox{ or }\beta\oplus \delta\neq
\eta}\chi(\mathbb{P}(\mathrm{Ext}^{1}(V_{\gamma},V_{\alpha})_{\xi}\times
\mathrm{Ext}^1(V_{\delta},V_{\beta})_{\eta}))g_{\gamma\delta}^{\xi'}g_{\alpha\beta}^{\eta'}\\
   &+& \int_{\alpha,\beta,\delta,\gamma,\alpha\oplus \gamma=\xi,
\beta\oplus\delta=
\eta}[d(\xi',\eta')-d(\gamma,\alpha)-d(\delta,\beta)-\lr{\gamma,\beta}]g_{\gamma\delta}^{\xi'}g_{\alpha\beta}^{\eta'}\\
   &-& \int_{\alpha,\beta,\delta,\gamma,\alpha\oplus \gamma=\xi,
\beta\oplus\delta=
\eta}\chi(\mathbb{P}\overline{\mathcal{V}}(V_{\xi},V_{\eta};V_{\xi'}\oplus
   V_{\eta'})).
\end{eqnarray*}

\end{Thm}
\section{Application to Caldero-Keller formula}\label{application}

\subsection{}
Let $Q$ be a quiver with vertex set $Q_0=\{1,2,\cdots, n\}$
containing no oriented cycles and $A=\bbc Q$ be the path algebra of
$Q.$ For $i\in Q_0,$ we denote by $P_i$ the corresponding
indecomposable projective $\bbc Q$-module and by $S_i$ the
corresponding simple module. Let $\mathbb{Q}(x_1,\cdots,x_n)$ be a
transcendental extension of $\mathbb{Q}.$ Define the map
$$
X_{?}:\mbox{obj}(\mathrm{mod}A)\rightarrow
\mathbb{Q}(x_1,\cdots,x_n)
$$
by:
$$
X_{M}=\sum_{\ue}\chi(Gr_{\ue}(M))x^{\tau(\ue) -
\underline{\mathrm{dim}}M + \ue}
$$
where $\tau$ is the Auslander-Reiten translation on the Grothendieck
group $K_0(\md^{b}(Q))$ and, for $v\in \mathbb{Z}^n$, we put
$$
x^v = \prod_{i=1}^n x_i^{\langle \underline{\mathrm{dim}}S_i,
v\rangle}
$$ and $Gr_{\ue}(M)$ is the $\ue$-Grassmannian of $M,$ i.e. the
variety of submodules of $M$ with dimension vector $\ue.$  This
definition is equivalent to \cite{Hubery2005}
$$
X_{M}=\int_{\alpha,\beta}g_{\alpha\beta}^{M}x^{\underline{\beta}
R+\underline{\alpha} R'-\underline{\mathrm{dim}}M}
$$
where the matrices $R=(r_{ij})$ and $R'=(r'_{ij})$ satisfy
$r_{ij}=\mathrm{dim}_{\bbc}\mathrm{Ext}^1(S_i,S_j)$ and
$r'_{ij}=\mathrm{dim}_{\bbc}\mathrm{Ext}^1(S_j,S_i)$ for $i,j\in
Q_0.$ Here we recall $g_{\alpha\beta}^M:=\chi(\mv((V_{\alpha},
V_{\beta}; M))$ which is defined in Section 4.1.
Note that (see \cite{Hubery2005})
$$
(\underline{\mathrm{\dim}}{P})R=\underline{\mathrm{\dim}}\rad P
\quad \quad
(\underline{\mathrm{\dim}}{I})R'=\underline{\mathrm{\dim}}I-\underline{\mathrm{\dim}}\mbox{soc}
I
$$

We consider the set
$$
Gr_{\ue}(\bbe_{\ud}):=\{(M, M_1)\mid M\in \bbe_{\ud}, M_1\in
Gr_{\ue}(M)\}.$$ This is a closed subset of
$\bbe_{\ud}\times\prod_{i=1,...,n}Gr_{e_i}(k^{d_i}).$
Here, we simply use the notation $\bbe_{\ud}$ instead of
$\bbe_{\ud}(Q)$ without confusion.

\begin{Prop}
The function $X_{?}\mid_{\bbe_{\ud}}$ is $G$-invariant
constructible.
\end{Prop}
\begin{proof} Obviously it is $G$-invariant.
Consider the canonical morphism $\pi:
Gr_{\ue}(\bbe_{\ud})\rightarrow \bbe_{\ud}$ sending $(M, M_1)$ to
$M.$ It is clear that $\pi^{-1}(M)=Gr_{\ue}(M).$ Let
$1_{Gr_{\ue}(\bbe_{\ud})}$ be the constant function on
$Gr_{\ue}(\bbe_{\ud}),$ by Theorem \ref{Joyce},
$(\pi)_{*}(1_{Gr_{\ue}(\bbe_{\ud})})$ is constructible. We know
that
$$
(\pi)_{*}(1_{Gr_{\ue}(\bbe_{\ud})})(M)=\chi(Gr_{\ue}(M))
$$
So there are finitely many $\chi(Gr_{\ue}(M))$ for $M\in
\bbe_{\ud}.$
\end{proof}
\begin{Prop}\label{finite}
For fixed dimension vectors $\ue$ and $\ud,$ the set
$$
\{g_{XY}^{E}\mid E\in \bbe_{\ud}, Y\in \bbe_{\ue},X\in
\bbe_{\ud-\ue}\}
$$
is a finite set.
\end{Prop}
\begin{proof}
Let $M\in \bbe_{\ud}$. For any submodule $M_1$ of dimension vector
$\ue$ of $M,$ by the knowledge of linear algebra, there exist
unique $(\bbc^{\ue},x)\in \bbe_{\ue}$ isomorphic to $M_1$ and
$(\bbc^{\ud-\ue}, x')\in \bbe_{\ud-\ue}$ isomorphic to $M/M_1,$
this deduces the following morphisms:
$$\xymatrix {Gr_{\ue}(M)\ar[r]^-{\pi_1} &\bbe_{\ue}\times
\bbe_{\ud-\ue}\times \bbe_{\ud}\ar[r]^-{\pi_2}&
\bigcup_{i}\phi_i(U_i)}
$$
where $\bbe_{\ue}\times \bbe_{\ud-\ue}\times
\bbe_{\ud}=\bigcup_{i}U_i$ is a finite stratification with respect
to the action of the algebraic group $G_{\ue}\times G_{\ud-\ue}$
and $\phi_i: U_i\rightarrow \phi_i(U_i)$ is the geometric quotient
for any $i,$ and $\pi_2=\bigcup_i\phi_i.$ For any $(Y,X,M)\in
\bbe_{\ue}\times \bbe_{\ud-\ue}\times \bbe_{\ud}$,$$
\chi((\pi_2\pi_1)^{-1}(\pi_2((Y,X,M))))=g_{XY}^{M}.$$ Consider the
constant function $1_{Gr_{\ue}(M)}$ on $Gr_{\ue}(M),$ by Theorem
\ref{Joyce}, $(\pi_2\pi_1)_{*}(1_{Gr_{\ue}(M)})$ is constructible.
Hence, there are finitely many $g_{XY}^{M}$ for $(X,Y,M)\in
\bbe_{\ue}\times \bbe_{\ud-\ue}\times \bbe_{\ud}.$
\end{proof}
\begin{Prop}\label{finite02}
For fixed $M\in
\bbe_{\underline{\xi'}},N\in\bbe_{\underline{\eta'}},$ the set
$$\{\chi(\mathrm{Ext}^1(M,N)_{E})\mid E\in \bbe_{\underline{\xi'}+\underline{\eta'}}\}$$
is a finite set.
\end{Prop}
\begin{proof}
Consider the morphism:
$$
\mathrm{Ext}^1(M,N)\xrightarrow{f}
\bbe_{\underline{\xi'}+\underline{\eta'}}\xrightarrow{g}
\bigcup_{j}\phi_j(V_j)
$$
where $\bbe_{\underline{\xi'}+\underline{\eta'}}=\bigcup_{j}V_j$
is a finite stratification with respect to the action of the
algebraic group $G_{\underline{\xi'}+\underline{\eta'}}$ and
$\phi_j: V_j\rightarrow \phi_j(V_j)$ is a geometric quotient for
any $i,$ and $f$ sends any extension to the middle term of the
extension. Here, $f$ is a morphism by identification between
$\mathrm{Ext}^{1}(M,N)$ and $E(M,N)$ at the beginning of Section
4.  The remaining discussion is almost the same as in Proposition
\ref{finite}. We omit it.
\end{proof}

\subsection{}
We now consider the cluster category, i.e. the orbit category
$\md^{b}(Q)/F$ with $F=[1]\tau^{-1},$ where $\tau$ is the
AR-translation of $\md^{b}(Q).$ Each
object $M$ in $\md^{b}(Q)/F$ can be uniquely decomposed into the
form: $M=M_0\oplus P_M[1]=M_0\oplus\tau P_M$ where $M_0\in\mod A$
and $P_M$ is projective in $\mod A.$ Now we can extend the map
$X_?$ as in \cite{CK2005}, see also \cite{Hubery2005}
by the rule: $X_{\tau P}=x^{\udim {P/\rad P}}$ for projective
$A$-module and $X_{M\oplus N}=X_{M}X_{N}.$ Then we have a
well-defined map $$X_?:
\mathrm{obj}(\md^{b}(Q)/F)\ra\bbq(x_1,\cdots,x_n).$$

Let $\overline{\bbe}_{\ud}$ be the orbit space of $\bbe_{\ud}$
under the action of $G_{\ud}.$ Note that all the integrals below
are over $\overline{\bbe}_{\ud}$ for some corresponding dimension
vector $\ud.$ Note also that in $\mod A$ we have $d(\gamma,
\alpha)=\dim\hom(V_{\gamma},V_{\alpha})$ and $d^1(\gamma,
\alpha)=\dim\ext(V_{\gamma},V_{\alpha}).$ We say that $P_0$ is the
projective direct summand of $V_{\xi'}$ if $V_{\xi'}\simeq
V'_{\xi'}\oplus P_0$ and no direct summand of $V'_{\xi'}$ is
projective.

The cluster algebra corresponding to the cluster category
$\md^{b}(Q)/F$ is the subalgebra of $\bbq(x_1,\cdots,x_n)$
generated by $\{X_{M}, X_{\tau P}|M\in\mod A, P\in\mod A \
\mbox{is projective}\ \}.$ The following theorem gives a
generalization of the cluster multiplication formula in
\cite{CK2005}. The idea of the proof follows the work
\cite{Hubery2005} of Hubery.

\begin{Thm}\label{clustertheorem}
(1) For any $A$-modules $V_{\xi'}$, $V_{\eta'}$ we have
$$\hspace{0cm}d^1(\xi', \eta')X_{V_{\xi'}} X_{V_{\eta'}} =\int_{\lambda\neq
\xi'\oplus \eta'} \chi(\mathbb{P}\mathrm{Ext}^1(V_{\xi'},
V_{\eta'})_{V_{\lambda}})
X_{V_{\lambda}}$$$$+\int_{\gamma,\beta,\iota}\chi(\mathbb{P}\mathrm{Hom}(V_{\eta'},\tau
V_{\xi'})_{V_{\beta}[1]\oplus \tau V_{\gamma}'\oplus
I_0})X_{V_{\gamma}}X_{V_{\beta}}x^{\underline{\mathrm{dim}}
\mathrm{soc}I_0}
$$
where  $I_0\in \iota$ is injective and
$V_{\gamma}=V_{\gamma}'\oplus P_0,$ $P_0$ is the projective direct
summand of $V_{\xi'}.$

(2) For any $A$-module $V_{\xi'}$ and  $P\in \rho$ is projective
Then
 $$d(\rho, \xi')X_{V_{\xi'}}x^{\underline{\mathrm{dim}}P/radP}=\int_{\delta,\iota'}\chi(\mathbb{P}\mathrm{Hom}(V_{\xi'},I)_{V_{\delta}[1]\oplus
I'})X_{V_{\delta}}x^{\underline{\mathrm{dim}}\mathrm{soc}I'}$$$$+\int_{\gamma,\rho'}\chi(\mathbb{P}\mathrm{Hom}(P,V_{\xi'})_{P'[1]\oplus
V_{\gamma}})X_{V_{\gamma}}x^{\underline{\mathrm{dim}}P'/radP'}
 $$
where $I=D \hom(P,A),$ and $I'\in \iota'$ injective, $P'\in \rho'$
projective.
\end{Thm}

\begin{proof}
 We set
$$ S_1:= \int_{\lambda\in
\overline{\bbe}_{\underline{\xi'}+\underline{\eta'}},\lambda\neq
\xi'\oplus \eta'
}\chi(\mathbb{P}\ext^1(V_{\xi'},V_{\eta'})_{V_{\lambda}})
X_{V_{\lambda}}
$$
By Proposition \ref{finite},
$$
S_1=\int_{\xi,\eta,\lambda\neq \xi'\oplus
\eta'}\chi(\mathbb{P}\ext^1(V_{\xi'},V_{\eta'})_{V_{\lambda}})g_{\xi\eta}^{\lambda}x^{\underline{\eta}
R+\underline{\xi} R'-(\underline{\xi'}+\underline{\eta'})}
$$
Using Theorem \ref{projgreen}, we have
\begin{eqnarray*}
 &&\int_{\xi,\eta,\lambda\neq \xi'\oplus
\eta'}\chi(\mathbb{P}\ext^1(V_{\xi'},V_{\eta'})_{V_{\lambda}})g_{\xi\eta}^{\lambda}x^{\underline{\eta}
R+\underline{\xi} R'-(\underline{\xi'}+\underline{\eta'})}=\\
   &&\int_{\alpha,\beta,\delta,\gamma,\xi,\eta,\alpha\oplus \gamma\neq \xi
\mbox{ or }\beta\oplus \delta\neq
\eta}\chi(\mathbb{P}(\ext^{1}(V_{\gamma},V_{\alpha})_{V_{\xi}}\times
\ext^1(V_{\delta},V_{\beta})_{V_{\eta}}))g_{\gamma\delta}^{\xi'}g_{\alpha\beta}^{\eta'}x^{\underline{\eta}
R+\underline{\xi} R'-(\underline{\xi'}+\underline{\eta'})}\\
   &+& \int_{\alpha,\beta,\delta,\gamma,\xi,\eta,\alpha\oplus \gamma=\xi,
\beta\oplus\delta=
\eta}[d(\xi',\eta')-d(\gamma,\alpha)-d(\delta,\beta)-\lr{\gamma,\beta}]g_{\gamma\delta}^{\xi'}g_{\alpha\beta}^{\eta'}x^{\underline{\eta}
R+\underline{\xi} R'-(\underline{\xi'}+\underline{\eta'})}\\
   &-& \int_{\alpha,\beta,\delta,\gamma,\xi,\eta,\alpha\oplus \gamma=\xi,
\beta\oplus\delta=
\eta}\chi(\mathbb{P}\overline{V}(V_{\xi},V_{\eta};V_{\xi'}\oplus
   V_{\eta'}))x^{\underline{\eta}
R+\underline{\xi} R'-(\underline{\xi'}+\underline{\eta'})}
\end{eqnarray*}
We come to simplify every term following \cite{Hubery2005}. For
fixed $\alpha,\beta,\delta,\gamma,$
$$
\int_{\xi,\eta,\alpha\oplus \gamma\neq \xi \mbox{ or }\beta\oplus
\delta\neq
\eta}\chi(\mathbb{P}(\ext^{1}(V_{\gamma},V_{\alpha})_{V_{\xi}}\times
\ext^1(V_{\delta},V_{\beta})_{V_{\eta}}))=d^1(\gamma,\alpha)+d^1(\delta,\beta)
$$

$$
\int_{\alpha,\beta,\delta,\gamma,\xi,\eta,\alpha\oplus \gamma\neq
\xi \mbox{ or }\beta\oplus \delta\neq
\eta}\chi(\mathbb{P}(\ext^{1}(V_{\gamma},V_{\alpha})_{V_{\xi}}\times
\ext^1(V_{\delta},V_{\beta})_{V_{\eta}}))g_{\gamma\delta}^{\xi'}g_{\alpha\beta}^{\eta'}x^{\underline{\eta}
R+\underline{\xi} R'-(\underline{\xi'}+\underline{\eta'})}
$$
$$
\hspace{-4cm}=\int_{\alpha,\beta,\delta,\gamma}[d^1(\gamma,\alpha)+d^1(\delta,\beta)])g_{\gamma\delta}^{\xi'}g_{\alpha\beta}^{\eta'}x^{\underline{\eta}
R+\underline{\xi} R'-(\underline{\xi'}+\underline{\eta'})}
$$
Moreover,
$$
d^1(\gamma,\alpha)+d^1(\delta,\beta)+d(\xi',\eta')-d(\gamma,\alpha)-d(\delta,\beta)-\lr{\gamma,\beta}=d^1(\xi',\eta')
+\lr{\delta,\alpha}
$$

 Hence,
\begin{eqnarray*}
 &&\int_{\xi,\eta,\lambda\neq \xi'\oplus
\eta'}\chi(\mathbb{P}\ext^{1}(V_{\xi'},V_{\eta'})_{V_{\lambda}})g_{\xi\eta}^{\lambda}x^{\underline{\eta}
R+\underline{\xi} R'-(\underline{\xi'}+\underline{\eta'})}=\\
   &+& \int_{\alpha,\beta,\delta,\gamma}[d^1(\xi',\eta')+\lr{\delta,\alpha}]g_{\gamma\delta}^{\xi'}g_{\alpha\beta}^{\eta'}x^{\underline{\eta}
R+\underline{\xi} R'-(\underline{\xi'}+\underline{\eta'})}\\
   &-& \int_{\alpha,\beta,\delta,\gamma,\xi,\eta,\alpha\oplus \gamma=\xi,
\beta\oplus\delta=
\eta}\chi(\mathbb{P}\overline{V}(V_{\xi},V_{\eta};V_{\xi'}\oplus
   V_{\eta'}))x^{\underline{\eta}
R+\underline{\xi} R'-(\underline{\xi'}+\underline{\eta'})}
\end{eqnarray*}
As for the last term, consider the following diagram, it may be
compared with diagram \eqref{decom}).
\begin{equation}
\bigcup_{\alpha,\beta,\gamma,\delta}\mv(V_{\alpha},V_{\beta};
V_{\eta'})\times
\mv(V_{\gamma},V_{\delta};V_{\xi'})\xrightarrow{j_1}\bigcup_{\xi,\eta}\mv(V_{\xi},V_{\eta};V_{\xi'}\oplus
V_{\eta'})
\end{equation}
sending $(V^1_{\eta'}\subseteq V_{\eta'},V^1_{\xi'}\subseteq
V_{\xi'})$ to $(V^1_{\xi'}\oplus V^1_{\eta'}\subseteq
V_{\xi'}\oplus V_{\eta'}).$ And
\begin{equation}
\bigcup_{\xi,\eta}\mv(V_{\xi},V_{\eta};V_{\xi'}\oplus
V_{\eta'})\xrightarrow{j_2}\bigcup_{\alpha,\beta,\gamma,\delta}\mv(V_{\alpha},V_{\beta};
V_{\eta'})\times \mv(V_{\gamma},V_{\delta};V_{\xi'})
\end{equation}
sending $(V^1\subseteq V_{\xi'}\oplus V_{\eta'})$ to $(V^1\bigcap
V_{\eta'}\subseteq V_{\eta'}, V^1/V^1\bigcap V_{\eta'}\subseteq
V_{\xi'}).$ The map $j_1$ is an embedding and
$$
\bigcup_{\xi,\eta}\mv(V_{\xi},V_{\eta};V_{\xi'}\oplus
V_{\eta'})\setminus
\mbox{Im}j_1=\bigcup_{\xi,\eta}\overline{\mv}(V_{\xi},V_{\eta};V_{\xi'}\oplus
   V_{\eta'})
$$
The fibre of $j_2$ is isomorphic to a vector space
$V(\delta,\alpha)$ of dimension $d(\delta,\alpha)$ ( see
\cite[Corollary 8]{Hubery2005}). If we restrict $j_2$ to
$\bigcup_{\xi,\eta}\overline{\mv}(V_{\xi},V_{\eta};V_{\xi'}\oplus
   V_{\eta'}),$ then the fibre is isomorphic to $V(\delta,\alpha)\setminus
   \{0\}.$ Under the action of $\bbc^*,$ by Proposition \ref{Euler},
   we have
\begin{eqnarray}
   && \int_{\alpha,\beta,\delta,\gamma,\xi,\eta,\alpha\oplus \gamma=\xi,
\beta\oplus\delta=
\eta}\chi(\mathbb{P}\overline{\mv}(V_{\xi},V_{\eta};V_{\xi'}\oplus
   V_{\eta'}))x^{\underline{\eta}
R+\underline{\xi} R'-(\underline{\xi'}+\underline{\eta'})}\nonumber \\
  &=& \int_{\alpha,\beta,\delta,\gamma}d(\delta,\alpha)g_{\gamma\delta}^{\xi'}g_{\alpha\beta}^{\eta'}x^{\underline{\eta}
R+\underline{\xi}
R'-(\underline{\xi'}+\underline{\eta'})}.\nonumber
\end{eqnarray}
Therefore,
$$
S_1=\int_{\alpha,\beta,\delta,\gamma}[d^1(\xi',\eta')-d^1(\delta,\alpha)]g_{\gamma\delta}^{\xi'}g_{\alpha\beta}^{\eta'}x^{\underline{\eta}
R+\underline{\xi} R'-(\underline{\xi'}+\underline{\eta'})}.
$$
There is a natural $\bbc^*$-action on $\hom(V_{\eta'},\tau
V_{\xi'})_{V_{\beta}[1]\oplus \tau V_{\gamma}'\oplus I_0}\setminus
\{0\}$ by left multiplication, the orbit space is
$\mathbb{P}\hom(V_{\eta'},\tau V_{\xi'})_{V_{\beta}[1]\oplus \tau
V_{\gamma}'\oplus I_0}$. Define
$$
S_2:=\int_{\gamma,\beta,\iota,\kappa,l,\mu,\theta}\chi(\mathbb{P}\hom(V_{\eta'},\tau
V_{\xi'})_{V_{\beta}[1]\oplus \tau V_{\gamma}'\oplus
I_0})g_{\kappa
l}^{\gamma}g_{\theta\mu}^{\beta}x^{(\underline{l}+\underline{\mu})R+(\underline{\kappa}+\underline{\theta})R'-(\underline{\beta}+\underline{\gamma})+\mathrm{\underline{dim}}\mathrm{soc}I_0}
$$
where $I_0\in \iota.$  The above definition is well-defined by
Proposition \ref{finite} and \ref{finite02}. We note that
$$
\underline{\mathrm{dim}}
\mathrm{soc}I_0-(\underline{\beta}+\underline{\gamma})=(\underline{\xi'}-\underline{\gamma})R+(\underline{\eta'}-
\underline{\beta})R'-(\underline{\xi}'+\underline{\eta}').
$$
Since $\hom(V_{\eta'},\tau V_{\xi'}')_{V_{\beta}[1]\oplus \tau
V_{\gamma}'\oplus I_0}=\mathcal{V}(\tau V_{\gamma}'\oplus
I_0,V_{\beta};V_{\eta'},\tau V_{\xi'}),$ we can apply Theorem
\ref{associativity2} to the following diagram twice:
\begin{equation}
\xymatrix{&V_{\mu}\ar[d]\ar@{=}[r]&V_{\mu}\ar[d]& \tau
V_{\kappa}\ar@{=}[r]&\tau
V_{\kappa}&\\
0\ar[r]& V_{\beta}\ar[r]\ar[d]& V_{\eta'}\ar[r]\ar[d]& \tau
V_{\xi'}'\ar[r]\ar@{=}[d]\ar[u]& \tau V_{\gamma}'\oplus
I_0\ar[r]\ar[u]\ar@{=}[d]& 0\\
0\ar[r]& X\ar[r]\ar@{=}[d]& V_{\widetilde{\beta}}\ar[r]\ar@{=}[d]&
\tau
V_{\xi'}'\ar[r] & \tau V_{\gamma}'\oplus I_0\ar[r]& 0\\
0\ar[r]& X\ar[r]& V_{\widetilde{\beta}}\ar[r]& \tau
V_{\widetilde{\gamma}}'\ar[r]\ar[u]& \tau L\oplus I_0
\ar[u]\ar[r]& 0}
\end{equation}
where $V_{\widetilde{\beta}}$ and $\tau V_{\widetilde{\gamma}}'$
are the corresponding pullback and pushout. Moreover,
$$
(\underline{\xi}'-\underline{\gamma}+\underline{l}+\underline{\mu})R+(\underline{\eta}'-\underline{\beta}+
\underline{\kappa}+\underline{\theta})R'-(\underline{\xi}'+\underline{\eta}')=(\underline{\widetilde{\gamma}}+
\underline{\mu})R+(\underline{\kappa}+\underline{\widetilde{\beta}})R'-(\underline{\xi}'+\underline{\eta}')
$$
Hence,
\begin{eqnarray}
  S_2 &=& \int_{\widetilde{\gamma},\widetilde{\beta},\kappa,\mu,\iota,l,\theta}\chi(\mathbb{P}\hom(V_{\widetilde{\beta}},\tau
V_{\widetilde{\gamma}}')_{X[1]\oplus \tau L\oplus I_0})g_{\kappa
\widetilde{\gamma}}^{\xi'}g_{\widetilde{\beta}\mu}^{\eta'}x^{(\underline{\widetilde{\gamma}}+\underline{\mu})R+(\underline{\kappa}+\underline{\widetilde{\beta}})R'-(\underline{\xi'}+\underline{\eta}')}
\nonumber \\
 &=&\int_{\widetilde{\gamma},\widetilde{\beta},\kappa,\mu}\chi(\mathbb{P}\hom(V_{\widetilde{\beta}},\tau
V_{\widetilde{\gamma}}'))g_{\kappa
\widetilde{\gamma}}^{\xi'}g_{\widetilde{\beta}\mu}^{\eta'}x^{(\underline{\widetilde{\gamma}}+\underline{\mu})R+(\underline{\kappa}+\underline{\widetilde{\beta}})R'-(\underline{\xi}'+\underline{\eta}')}
\nonumber \\
  &=&\int_{\alpha,\beta,\delta,\gamma}d^1(\delta,\alpha)g_{\gamma\delta}^{\xi'}g_{\alpha\beta}^{\eta'}x^{\underline{\eta}
R+\underline{\xi}
R'-(\underline{\xi'}+\underline{\eta'})}.\nonumber
\end{eqnarray}
Hence,
$$
S_1+S_2=d^1(\xi',\eta')\int_{\alpha,\beta,\delta,\gamma}g_{\gamma\delta}^{\xi'}g_{\alpha\beta}^{\eta'}x^{\underline{\eta}
R+\underline{\xi} R'-(\underline{\xi}'+\underline{\eta}')}
$$
The first assertion is proved. In order to prove the second
part, by Theorem \ref{associativity2}, we have
$$\int_{\delta,\delta_1,\delta_2,\iota'}g_{\delta_1\delta_2}^{\delta}\chi(\mathbb{P}\hom(V_{\xi'},I)_{V_{\delta}[1]\oplus
I'})x^{\underline{\delta}_2R+\underline{\delta}_1R'-\underline{\delta}+\mathrm{\underline{dim}}\mathrm{soc}I'}
$$
$$=\int_{\widetilde{\xi'},\delta_1,\delta_2,\iota'}g_{\widetilde{\xi'}\delta_2}^{\xi'}\chi(\mathbb{P}\hom(V_{\widetilde{\xi'}},I)_{V_{\delta_1}[1]\oplus
I'})x^{\underline{\delta_2}R+\underline{\widetilde{\xi'}}R'-\underline{\xi'}+\mathrm{\underline{dim}}\mathrm{soc}I}$$
$$
\hspace{-2cm}=\int_{\widetilde{\xi'},\delta_2}g_{\widetilde{\xi'}\delta_2}^{\xi'}\chi(\mathbb{P}\hom(V_{\widetilde{\xi'}},I))x^{\underline{\delta_2}R+\underline{\widetilde{\xi'}}R'-\underline{\xi'}+\mathrm{\underline{dim}}\mathrm{soc}I}
$$
and
$$\int_{\gamma,\gamma_1,\gamma_2,\rho'}g_{\gamma_1\gamma_2}^{\gamma}\chi(\mathbb{P}\hom(P,V_{\xi'})_{P'[1]\oplus
V_{\gamma}})x^{\underline{\gamma_2}R+\underline{\gamma_1}R'-\underline{\gamma}+\mathrm{\underline{dim}}P'/radP'}$$
$$=\int_{\widetilde{\xi'},\gamma_1,\gamma_2,\rho'}g_{\gamma_1\widetilde{\xi'}}^{\xi'}\chi(\mathbb{P}\hom(P,V_{\widetilde{\xi'}})_{P'[1]\oplus
V_{\gamma_2}})x^{\underline{\gamma_1}R+\underline{\widetilde{\xi'}}R'-\underline{\xi}'+\mathrm{\underline{dim}}P/radP}$$
$$
\hspace{-2cm}=\int_{\widetilde{\xi'},\gamma_1}g_{\widetilde{\xi'}\gamma_1}^{\xi'}\chi(\mathbb{P}\hom(P,V_{\widetilde{\xi'}}))x^{\underline{\gamma_1}R+\underline{\widetilde{\xi'}}R'-\underline{\xi}'+\mathrm{\underline{dim}}P/radP}.
$$
We note that $$\mathrm{\underline{dim}}\ \mathrm{soc}I=
\mathrm{\underline{dim}}\ P/radP$$ and
$$
\chi(\mathbb{P}\hom(P,M))=\chi(\mathbb{P}\hom(P,V_{\widetilde{\xi'}}))+\chi(\mathbb{P}\hom(V_{\widetilde{\xi'}},I)).
$$ The second assertion is proved.
\end{proof}

\subsection{}
We illustrate Theorem \ref{clustertheorem} by the following
example.

Let $Q$ be the Kronecker quiver $\xymatrix{1\ar @<2pt>[r]
\ar@<-2pt>[r]& 2}.$ Let $S_1$ and $S_2$ be the simple modules
associated to
vertices 1 and 2, respectively. Hence,
$$
R=\left(%
\begin{array}{cc}
  0 & 2 \\
  0 & 0 \\
\end{array}%
\right) \quad \mbox{and}\quad R'=\left(%
\begin{array}{cc}
  0 & 0 \\
  2 & 0 \\
\end{array}%
\right)
$$
and
$$
X_{S_1}=x^{\mathrm{\underline{dim}}S_1R'-\mathrm{\underline{dim}}S_1}+x^{\mathrm{\underline{dim}}S_1R-\mathrm{\underline{dim}}S_1}=x_1^{-1}(1+x_2^2),
$$
$$
X_{S_2}=x^{\mathrm{\underline{dim}}S_2R'-\mathrm{\underline{dim}}S_2}+x^{\mathrm{\underline{dim}}S_2R-\mathrm{\underline{dim}}S_2}=x_2^{-1}(1+x_1^2).
$$
For $\lambda\in \mathbb{P}^1(\bbc),$ let $u_{\lambda}$ be the
regular representation $\xymatrix{\bbc\ar @<2pt>[r]^{1}
\ar@<-2pt>[r]_{\lambda}& \bbc}.$ Then
$$
X_{u_{\lambda}}=x^{(1,1)R'-(1,1)}+x^{(1,1)R-(1,1)}+x^{(0,1)R+(1,0)R'-(1,1)}=x_1x_2^{-1}+x_1^{-1}x_2
+x_1^{-1}x_2^{-1}.$$ Let $I_1$ and $I_2$ be the indecomposable
injective modules corresponding vertices 1 and 2, respectively,
then
$$
X_{(I_1\oplus I_2)[-1]}:=x^{\mathrm{\underline{dim}soc}(I_1\oplus
I_2)}=x_1x_2
$$
The left side of the identity of Theorem \ref{clustertheorem} is
$$
\mathrm{dim}_{\bbc}\ext^{1}(S_1,S_2)X_{S_1}X_{S_2}=2(x_1^{-1}x_2^{-1}+x_1x_2^{-1}+x_1^{-1}x_2+x_1x_2).
$$
The first term of the right side is
$$
\int_{\lambda\in
\mathbb{P}^1(\bbc)}\chi(\mathbb{P}\ext^1(S_1,S_2)_{u_{\lambda}})X_{u_{\lambda}}=2(x_1^{-1}x_2^{-1}+x_1x_2^{-1}+x_1^{-1}x_2).
$$
To compute the second term of the right side, we note that for any
$f\neq 0\in \hom(S_2,\tau S_1),$ we have the following exact
sequence:
$$
0\rightarrow S_2\xrightarrow{f}\tau S_1\rightarrow I_1\oplus
I_2\rightarrow 0.
$$
This implies $\hom(S_2,\tau S_1)\setminus \{0\}=\hom(S_2,\tau
S_1)_{I_1\oplus I_2}$ Hence, the second term is equal to
$2x_1x_2.$

\end{document}